\documentclass[a4paper]{amsart}
\usepackage[utf8]{inputenc}
\usepackage{enumerate}
\usepackage{enumitem}
\usepackage{mathtools}
\usepackage{amsthm}
\usepackage{amsmath}
\usepackage{mathrsfs}
\usepackage{array}
\usepackage{fancyhdr}
\usepackage{bbm}
\usepackage{faktor}
\usepackage{fouriernc}
\usepackage{amssymb}
\usepackage{amsfonts}
\usepackage{abstract}
\usepackage[all,cmtip,2cell]{xy}
\usepackage[left=1in + \hoffset, right=1in +\hoffset
]{geometry}

\usepackage[backend=bibtex,style=alphabetic,sorting=nty,doi=false,url=false,eprint=false]{biblatex}
\usepackage{fancyhdr}
\newcommand{\blank}{\underline{\hspace{0.2cm}}}
\def\bign#1{\mathclose{\hbox{$\left#1\vbox to8.5\p@{}\right.\n@space$}}\mathopen{}}
\swapnumbers
\theoremstyle{definition}
\newtheorem{theorem}{Theorem}[]
\newtheorem*{theorem*}{Theorem}
\newtheorem{lemma}[theorem]{Lemma}
\newtheorem*{lemma*}{Lemma}
\newtheorem{corollary}[theorem]{Corollary}
\newtheorem*{corollary*}{Corollary}
\newtheorem{definition}[theorem]{Definition}
\newtheorem*{definition*}{Definition}

\newtheorem*{remark*}{Remark}
\newtheorem{proposition}[theorem]{Proposition}
\newtheorem*{proposition*}{Proposition}
\newtheorem{example}[theorem]{Example}
\newtheorem*{example*}{Example}

\newcommand\restr[2]{\ensuremath{\left.#1\right|_{#2}}}
\DeclareMathOperator*\medoplus{\mathchoice
	{\textstyle\bigoplus}
	{\textstyle\bigoplus}
	{\scriptstyle\bigoplus}
	{\scriptscriptstyle\bigoplus}
}
\title{The $\operatorname{Ext}$-algebra of standard modules over dual extension algebras}
\author{Markus Thuresson}

\begin{document}
\maketitle
\thispagestyle{empty}
\begin{abstract}
We exhibit an isomorphism of associative algebras between the $\operatorname{Ext}$-algebra $\operatorname{Ext}_\Lambda^\ast(\Delta,\Delta)$ of standard modules over the dual extension algebra $\Lambda$ of two directed algebras $B$ and $A$ and the dual extension algebra of the $\operatorname{Ext}$-algebra $\operatorname{Ext}_B^\ast(\mathbb{L},\mathbb{L})$ with $A$. There are natural $A_\infty$-structures on these $\operatorname{Ext}$-algebras, and, under certain technical assumptions on $B$, we describe that on $\operatorname{Ext}_\Lambda^\ast(\Delta,\Delta)$ completely in terms of that on $\operatorname{Ext}_B^\ast(\mathbb{L},\mathbb{L})$. As an example, we compute these $A_\infty$-structures explicitly in the case where $B=A=K\mathbb{A}_n /(\operatorname{rad}K\mathbb{A}_n)^\ell$.
\end{abstract}
\newpage
\thispagestyle{empty}

\noindent
\section{Introduction}
In \cite{CPS}, Cline, Parshall and Scott introduced the notion of highest weight category as a categorical axiomatization of structures arising in the representation theory of complex semisimple Lie algebras. Moreover, they showed that those finite-dimensional algebras whose module categories are equivalent to a highest weight category are exactly the quasi-hereditary algebras. Typical examples of quasi-hereditary algebras are hereditary algebras, algebras of global dimension two, Schur algebras and blocks of BGG category $\mathcal{O}$. The defining feature of quasi-hereditary algebras is the existence of a particular collection of modules, called the standard modules. These are certain quotients of the indecomposable projective modules and serve as the main protagonists of the representation theory of quasi-hereditary algebras. Closely related is the associated category $\mathcal{F}(\Delta)$, the full subcategory of the module category consisting of those modules which admit a filtration by standard modules. 

An important step towards understanding $\mathcal{F}(\Delta)$ for a general quasi-hereditary algebra was taken by Koenig, K{\"u}lshammer and Ovsienko in \cite{kko}. Using powerful techniques involving $A_\infty$-algebras and boxes, they showed that for any quasi-hereditary algebra, there is a directed box such that the category of representations of this box is equivalent to $\mathcal{F}(\Delta)$, allowing the study of $\mathcal{F}(\Delta)$ through the study of boxes and their representations (a box is a bimodule over an algebra together with a comultiplication and counit obeying the natural coassociativity and counitality axioms \cite{Rojter}). The representation theory of boxes is important in its own right, notably playing a sizeable role in the proof of Drozd's theorem on tame and wild dichotomy \cite{Drozd}.

The result by Koenig, K{\"u}lshammer and Ovsienko leads us to the problem of, given a quasi-hereditary algebra, determining the corresponding box describing $\mathcal{F}(\Delta)$. This may be done by taking the so-called  ``$A_\infty$-Koszul dual'' of the algebra of extensions between standard modules, $\operatorname{Ext}^\ast(\Delta,\Delta)$. The first (and arguably most arduous) step in this process is to determine the $A_\infty$-structure on $\operatorname{Ext}^\ast(\Delta,\Delta)$. Loosely speaking, such a structure is meant to capture the idea of an algebra whose multiplication is not strictly associative, but only associative up to a system of higher homotopies \cite{Stasheff1, Stasheff2}.

Unfortunately, describing the $A_\infty$-structure on $\operatorname{Ext}^\ast(\Delta,\Delta)$ may be extremely complicated. Examples of algebras where this $A_\infty$-structure has been explicitly computed are few and far between, however, one family of examples may be found in \cite{KlamtStroppel}. The main goal of the present paper is to study a somewhat large class of algebras which are computationally well-behaved enough to permit an explicit description.

In \cite{Xi}, Xi introduced the dual extension algebra as part of his study of BGG algebras, that is, quasi-hereditary algebras adimitting a simple-preserving duality on their module category. Originally introduced as taking only one input ($\mathcal{A}(B,A^{\operatorname{op}})$ with $B=A$), these were soon generalized and further studied in \cite{DengXi, Xigldim, DengXiringel, XiTilting, Shun, LiWei, Lixu,}.  Importantly, in \cite{Lixu}, Li and Xu connected the Koszulity of the dual extension algebra $\mathcal{A}(B,A^{\operatorname{op}})$ to that of $B$ and $A$. The dual extension algebras, as it turns out, behave well enough with respect to computations to allow the explicit description of the $A_\infty$-structure mentioned above.

The following is a description of the main results of the present article. Let $B$ and $A$ be directed algebras and let $\Lambda=\mathcal{A}(B,A^{\operatorname{op}})$ denote their dual extension algebra. In this case, $\Lambda$ is quasi-hereditary by a result of Xi. Let $\Delta$ denote the direct sum of standard modules over $\Lambda$ (one from each isomorphism class) and let $\mathbb{L}$ denote the direct sum of simple modules over $B$ (one from each isomorphism class).

\begin{enumerate}[label=(\Alph*)]
	\item There is an isomorphism of graded algebras between the $\operatorname{Ext}$-algebra of standard modules over the dual extension algebra $\Lambda$ and the dual extension algebra of the $\operatorname{Ext}$-algebra of simple modules over $B$ with $A$, that is, 
	\begin{align*}
	\operatorname{Ext}^\ast_{\Lambda}(\Delta,\Delta)\cong \mathcal{A}(\operatorname{Ext}_B^\ast(\mathbb{L},\mathbb{L}), A).
	\end{align*}
	\item The special case of Merkulov's construction (\cite{Merkulov}) given in \cite{LuPalmieriWuZhang} provides $A_\infty$-structures on $\operatorname{Ext}_B^\ast(\mathbb{L},\mathbb{L})$ and on $\operatorname{Ext}_{\Lambda}^\ast(\Delta,\Delta)$. Denote the higher multiplications by $\{m_n^B\}_{n=1}^\infty$ and $\{m_n^\Lambda\}_{n=1}^\infty$, respectively. Merkulov's construction involves several choices and hence the higher multiplications produced are not canonical. We show that, under certain technical assumptions on $B$, the construction may be performed on $\operatorname{Ext}_B^\ast(\mathbb{L},\mathbb{L})$ and $\operatorname{Ext}_{\Lambda}^\ast(\Delta,\Delta)$ in such a way that the higher multiplications $\{m_n^\Lambda \}_{n=1}^\infty$ are given in terms of the data produced by performing the construction on $\operatorname{Ext}_B^\ast(\mathbb{L},\mathbb{L})$. 
	
	More precisely, we prove the formulae below. For all details on the notation, and the assumptions needed, we refer to Section~6. The maps $p^B, h^B$ and $\lambda_{n-1}^B$ are obtained from performing Merkulov's construction on $\operatorname{Ext}_B^\ast(\mathbb{L},\mathbb{L})$, while $F$ denotes the induction functor
	$$\Lambda\otimes_B \blank:B\operatorname{-mod}\to \Lambda\operatorname{-mod}.$$ Let $n\geq 2$. For any $f^\prime_1,\dots, f_n^\prime \in \operatorname{Hom}_\Lambda(\Delta,\Delta)$ and $\varepsilon_1,\dots, \varepsilon_n\in \operatorname{Ext}_B^\ast(\mathbb{L},\mathbb{L})$, such that $\deg \varepsilon_i \geq 1$, for all $1\leq i \leq n$, we have the following.
	\begin{enumerate}[label=$(\roman*)$]
		\item If there is $1\leq i\leq n-1$ such that $f_i^\prime\in \operatorname{rad}(\Delta_\Lambda(\texttt{j}),\Delta_\Lambda(\texttt{k}))$, we have $m_n^\Lambda(f_n^\prime \varepsilon_n, \dots, f_1^\prime \varepsilon_1)=0$.
		\item $
		m_n^\Lambda(f_n^\prime \varepsilon_n, \dots, \varepsilon_2, \varepsilon_1)=(-1)^{n+1} f_n ^\prime F \big(p^B \varepsilon_n h^B (\lambda^B_{n-1}(\varepsilon_{n-1},\dots, \varepsilon_1))\big).$
	\end{enumerate}
	\item Lastly, we provide formulae for the $A_\infty$-multiplications on $\operatorname{Ext}_B^\ast(\mathbb{L},\mathbb{L})$ obtained from Merkulov's construction in the case  $B=\faktor{K\mathbb{A}_n}{(\operatorname{rad}K\mathbb{A}_n)^\ell}$ and, using (B), explicitly describe the corresponding $A_\infty$-multiplications that this gives on $\operatorname{Ext}^\ast_{\mathcal{A}(B,B^{\operatorname{op}})}(\Delta,\Delta)$.
\end{enumerate}
This article is organized in the following way. In Section 2, we fix some notation and recall the necessary definitions, as well as Xi's initial result on the quasi-hereditary structure of the dual extension algebra. In Section 3, we compute the space of extensions between standard modules over the dual extension algebra $\mathcal{A}(B,A^{\operatorname{op}})$. Section 4 is devoted to the description of the algebra structure on the $\operatorname{Ext}$-algebra of standard modules over $\mathcal{A}(B,A^{\operatorname{op}})$ and contains the proof of (A).

In Section 5, we provide an analogue of the theorem by Li and Xu in \cite{Lixu}, which states that $\mathcal{A}(B,A)$ is Koszul if and only if both $B$ and $A$ are Koszul, in terms of linear resolutions of standard modules. In Section 6, we investigate the $A_\infty$-structure on $\operatorname{Ext}_{\mathcal{A}(B,A^{\operatorname{op}})}^\ast(\Delta,\Delta)$ provided by Merkulov's construction and precisely state and prove (B).

Finally, in Sections 7 and 8, we give an example by performing this construction for $$B=A=\faktor{K\mathbb{A}_n}{(\operatorname{rad}K\mathbb{A}_n)^\ell}$$ and using the results of Section 6 to give an application to the dual extension algebra $\mathcal{A}(B,B^{\operatorname{op}})$.

\newpage
\tableofcontents
\newpage
\section{Notation and setup}
Throughout, we let $K$ be an algebraically closed field. Letting $Q=(Q_0,Q_1)$ be a quiver with vertex set $\{1,\dots, n\}$, we denote the path algebra of $Q$ by $KQ$. Throughout, all quivers will have vertex set $\{1, \dots, n\}$. This set inherits a natural ordering relation, ``$<$'', from the natural numbers, which we fix from here on. For an arrow $\alpha \in Q_1$, denote by $s(\alpha)$ and $t(\alpha)$ the starting and terminal vertex of $\alpha$, respectively. For vertices and arrows $\xymatrixcolsep{0.3cm}\xymatrix{i\ar[r]^-{\alpha} & j \ar[r]^-{\beta} & k}$ we write the composition ``$\beta$ after $\alpha$'' as $\beta\alpha$. For a path $p=\alpha_n \dots \alpha_1$ in $Q$, we write $s(p)=s(\alpha_1)$ and $t(p)=\alpha_n$. Given an admissible ideal $I\subset KQ$, we may form the corresponding quotient algebra $B=\faktor{KQ}{I}$. Let $L_B(i)$ and $P_B(i)$ denote the simple and indecomposable projective $B$-modules, respectively.

If the quiver $Q$ is finite and acyclic, we may assume that for any arrow $\xymatrixcolsep{0.3cm}\xymatrix{i\ar[r]^-{\alpha}& j}$, we have $i<j$ (with respect to the natural order on $\{1,\dots, n\}$). We say that $B=\faktor{KQ}{I}$, where $I$ is admissible, is \emph{directed} if $Q$ is finite, acyclic and its vertices are numbered as above. Of course, if $B$ instead is an algebra whose quiver only has arrows in decreasing direction, we may reverse the direction of arrows to obtain a directed algebra in the above sense.

\begin{definition}\label{def:dual extension algebra}\cite{Xi, XiTilting, Lixu}
	Let $B\cong \faktor{KQ}{I}$ and $A\cong\faktor{KQ^\prime}{I^\prime}$ be algebras and let the quivers $Q$ and $Q^\prime$ be such that $Q_0=Q_0^\prime$. We define the \emph{dual extension algebra} $\Lambda=\mathcal{A}(B,A)$ of $B$ and $A$ as $\mathcal{A}(B,A)=\faktor{KE}{J}$ where $E=(E_0,E_1)$ and $J$ are as follows.
	\begin{enumerate}[label=$(\roman*)$]
		\item $E_0=Q_0=Q_0^\prime$.
		\item $E_1=Q_1\cup Q_1^\prime$.
		\item If $I=\langle \rho_i \rangle$, $I^\prime=\langle \rho_j^\prime\rangle$, then
		$$J=\langle \rho_i, \rho_j^\prime, \alpha\beta^\prime \ |\forall \alpha\in Q_1, \forall \beta^\prime\in Q_1^\prime\rangle.$$
	\end{enumerate}
\end{definition}

It is clear that both $B$ and $A$ occur in a natural way as subalgebras as well as quotients of $\mathcal{A}(B,A)$. Importantly, this fact allows us to view modules over $B$ and $A$ as modules over $\Lambda$. Letting $p_B:\Lambda \to B$ be the natural surjection, we define the action of $a\in \Lambda$ on a $B$-module $M$ by $a\cdot m\coloneqq p_B(a)\cdot_B m$, for all $m\in M$, where $\cdot_B$ denotes the action of $B$ on $M$. For a $B$-module, this coincides with extending the action of $B$ to an action of $\Lambda$ by letting all arrows in the quiver of $\Lambda$ which come from the quiver of $A$ act as 0. Of course, a similar idea works for $A$-modules. Moreover, there are functors
$$F\coloneqq \Lambda\otimes_B \blank: B\operatorname{-mod}\to \Lambda\operatorname{-mod}\quad \textrm{and}\quad G\coloneqq B\otimes_\Lambda\blank: \Lambda\operatorname{-mod}\to B\operatorname{-mod},$$
which will be of importance.
In the above definition, the set $\{1,\dots, n\}$ indexes isomorphism classes of simple modules over $B$ as well as over $A$. Note that $\{1\dots, n\}$ also indexes isomorphism classes of simple modules over $\Lambda=\mathcal{A}(B, A)$. Consider the following quotient of the indecomposable projective $P_B(i)$, called the \emph{standard module} at $i$. 
$$\Delta_B(i)=\faktor{P_B(i)}{\sum_{f:P_B(j)\to P_B(i)} \operatorname{im}f}.$$
Here, the sum is taken over all homomorphisms $f:P_B(j)\to P_B(i)$ such that $i<j$. The algebra $B$ is said to be \emph{quasi-hereditary} the following hold.
\begin{enumerate}[label=$(\roman*)$]
	\item $\operatorname{End}_B (\Delta_B(i))\cong K$ for each $i\in \{1,\dots, n\}$.
	\item The indecomposable projectives $P_B(i)$ are filtered by standard modules, i.e., each projective $P_B(i)$ admits a chain of submodules
	$$0\subset M_0 \subset M_1 \subset \dots \subset M_\ell=P_B(i)$$
	such that all subquotients $\faktor{M_k}{M_{k-1}}$ are standard modules.
\end{enumerate}
\begin{definition}\cite[Definition~3.4]{konig1995exact, bkk}\label{definition:exact borel subalg}
	Let $(\Lambda,<)$ be a quasi-hereditary algebra with $n$ simple modules, up to isomorphism. Then, a subalgebra $B\subset \Lambda$ is called an \emph{exact Borel subalgebra} provided that
	\begin{enumerate}[label=$(\roman*)$]
		\item $B$ also has $n$ simple modules up to isomorphism and $(B,<)$ is directed,
		\item the functor $\Lambda\otimes_B\blank$ is exact, and
		\item there are isomorphisms $\Lambda\otimes_B L_B(i)\cong \Delta_\Lambda(i).$
	\end{enumerate}
If, in addition, the map $\operatorname{Ext}_B^k(L_B(i),L_B(j))\to \operatorname{Ext}_\Lambda^k(\Lambda\otimes_B L_B(i), \Lambda\otimes_B L_B(j))$ induced by the functor $\Lambda\otimes_B\blank$ is an isomorphism for all $k\geq 1$ and $i,j\in\{1,\dots, n\}$, $B\subset \Lambda$ is called a \emph{regular exact Borel subalgebra}. 
\end{definition}
\begin{theorem}\label{theorem:dual ext alg is qh} \cite[Example~1.6]{Xi}
	Let $B\cong \faktor{KQ}{I}$ and $A\cong \faktor{KQ^\prime}{I^\prime}$ be directed algebras. Then $\Lambda=\mathcal{A}(B,A^{\operatorname{op}})$ is quasi-hereditary with respect to the natural ordering on $\{1,\dots, n\}$. Moreover, there are isomorphisms of $\Lambda$-modules $\Delta_\Lambda(i)\cong P_{A^{\operatorname{op}}}(i)$ for all $i\in \{1,\dots, n\}$.
	\end{theorem}
\begin{example}\label{example: dual extension algebra}
	Consider the quiver $Q=\xymatrixcolsep{0.3cm}\xymatrix{1\ar[r]^-\alpha & 2\ar[r]^-\beta & 3}$ and put $B=A=KQ$. Then, the dual extension algebra $\Lambda=\mathcal{A}(B,B^{\operatorname{op}})$ is given by the quiver
		$$\xymatrixcolsep{0.3cm}\xymatrix{1\ar@/^.5pc/[r]^-{\alpha} & 2 \ar@/^.5pc/[r]^-{\beta}\ar@/^.5pc/[l]^-{\alpha^\prime} &3 \ar@/^.5pc/[l]^-{\beta\prime}}$$
	
subject to the relations $\alpha \alpha^\prime=0$ and $\beta\beta^\prime=0$. The indecomposable projective modules over $\Lambda$ have Loewy diagrams
\[P_\Lambda(1):\vcenter{\xymatrixcolsep{0.3cm}\xymatrixrowsep{0.3cm}\xymatrix{
& &1 \ar[ld]	\\
& 2 \ar[ld]\ar[rd] & \\
3 \ar[rd] & & 1\\
& 2 \ar[rd]\\
& & 1
}},\quad P_\Lambda(2): \vcenter{\xymatrixcolsep{0.3cm}\xymatrixrowsep{0.3cm}\xymatrix{
& 2 \ar[ld]\ar[rd] & \\
3 \ar[rd] & & 1\\
& 2 \ar[rd]\\
& & 1
}},\quad \textrm{and}\quad P_\Lambda(3): \vcenter{\xymatrixcolsep{0.3cm}\xymatrixrowsep{0.3cm}\xymatrix{
3 \ar[rd]\\
& 2 \ar[rd]\\
& & 1
}}.\]
We see that, then, the standard modules over $\Lambda$ have Loewy diagrams
\[\Delta_\Lambda(1) \cong L_\Lambda(1): \xymatrix{1},\quad \Delta_\Lambda(2): \vcenter{ \xymatrixrowsep{0.3cm}\xymatrix{2 \ar[d]\\1}},\quad \textrm{and}\quad \Delta_\Lambda(3)\cong P_\Lambda(3):\vcenter{ \xymatrixrowsep{0.3cm}\xymatrix{3\ar[d]\\2\ar[d]\\1}}\]
which, as we expect, coincide with the Loewy diagrams of the indecomposable projective modules over $B^{\operatorname{op}}$.
\end{example}
\section{Extensions between standard modules}
In this section, we compute $\operatorname{Ext}_\Lambda^k(\Delta_\Lambda(i),\Delta_\Lambda(j))$ for $B,A$ directed and $\Lambda=\mathcal{A}(B,A^{\operatorname{op}})$. We have the following crucial lemma, which is a collection of observations found in \cite[Lemma~1.6]{DengXi} and \cite[Lemma~2]{Shun}.

\begin{lemma}\label{lemma:regular tensor simple iso standard & regular tensor projective iso projective,tensor preserves projective cover,lemma:tensor functors invertible on objects} 
\begin{enumerate}
	\item The functor $F$ is exact.
	\item  There are isomorphisms of left $\Lambda$-modules $F(L_B(i))\cong \Delta_\Lambda(i)$ and $F(P_B(i))\cong P_\Lambda(i).$
	\item  For any $M\in B\operatorname{-mod}$, $\xymatrix@=0.3cm{P^\bullet \ar[r]& M}$ is a minimal projective resolution if and only if $\xymatrix@=0.3cm{F(P^\bullet)\ar[r]& F(M)}$ is a minimal projective resolution.
	\item The functors $G\circ F$ and $\operatorname{Id}_{B\operatorname{-mod}}$ are naturally isomorphic.
\end{enumerate}
\end{lemma}
It is now clear that $B\subset \mathcal{A}(B,A^{\operatorname{op}})$ is an exact Borel subalgebra. This was already noticed by Xi in \cite{Xi}. However, we will see in Example \ref{B in dual extension not generally regular}, that $B\subset \mathcal{A}(B,A^{\operatorname{op}})$ is, in general, not regular.

Next, we want to compute the spaces $\operatorname{Ext}_\Lambda^k(\Delta_\Lambda(i),\Delta_\Lambda(j))$ for general $i, j\in\{1,\dots, n\}$, by applying the functor $\operatorname{Hom}_{\Lambda}(\blank,\Delta_\Lambda(j))$ to a minimal projective resolution of $\Delta_\Lambda(i)$. Observe that Lemma \ref{lemma:regular tensor simple iso standard & regular tensor projective iso projective,tensor preserves projective cover,lemma:tensor functors invertible on objects} implies that $\xymatrix@=0.3cm{P^\bullet \ar[r] &L_B(i)}$ is a minimal projective resolution if and only if $\xymatrix@=0.3cm{F(P^\bullet)\ar[r]& \Delta_\Lambda(i)}$ is. This motivates the investigation of the spaces $\operatorname{Hom}_\Lambda(P_\Lambda(i),\Delta_\Lambda(j))$, about which we make the following observation.
\begin{proposition}\label{proposition:homs from projective to standard}
 Let $\pi_i:P_\Lambda( i)\to P_{A^{\operatorname{op}}}(i)$ be the natural projection. Then, there are isomorphisms of vector spaces
\begin{align*}\operatorname{Hom}_\Lambda\left(P_\Lambda(i), \Delta_\Lambda(j)\right)\cong \operatorname{Hom}_{\Lambda}\left(P_{A^{\operatorname{op}}}(i),P_{A^{\operatorname{op}}}(j)\right)\cong \operatorname{Hom}_{A^{\operatorname{op}}}(P_{A^{\operatorname{op}}}(i),P_{A^{\operatorname{op}}}(j))
\end{align*}
for all $i,j,\in\{1,\dots,n\}$,	given by precomposition with $\pi_i$.
\end{proposition}
\begin{proof} We prove the existence of the first isomorphism, as the existence of the second is immediate. Since $\Delta_\Lambda(j)$ and $P_{A^{\operatorname{op}}}(j)$ are isomorphic as $\Lambda$-modules, it suffices to exhibit an isomorphism
	\begin{align*}
	\operatorname{Hom}_\Lambda(P_\Lambda(i), P_{A^{\operatorname{op}}}(j))\cong \operatorname{Hom}_{\Lambda}(P_{A^{\operatorname{op}}}(i),P_{A^{\operatorname{op}}}(j)).
	\end{align*} The module $P_\Lambda(i)$ has a basis given by the set
	$\{q^\prime p \mid s(p)=e_i, t(p)=s(q^\prime)\}$,
	where $p$ is a path in $B$ and $q^\prime$ is a path in $A^{\operatorname{op}}$. Firstly, note that if $q^\prime p$ is such that $p\neq e_i$, then $q^\prime p\in \ker \pi_i$. Clearly, such elements form a basis of $\ker \pi_i$. Let $\varphi: P_\Lambda(i)\to P_{A^{\operatorname{op}}}(j)$ be a homomorphism (of $\Lambda$-modules). We claim that, then, $q^\prime p\in \ker \varphi$ so that $\ker \pi_i\subset \ker \varphi$. Indeed, we have
	\[\varphi(q^\prime p)=\varphi(q^\prime p e_i)=q^\prime p \varphi(e_i)=0\]
	since $\varphi(e_i)$ is an element of $P_{A^{\operatorname{op}}}(j)$ and $p$, being a non-trivial path in $B$, acts as 0 on this module. The fact that $\ker \pi_i \subset \ker \varphi$ implies that the homomorphism $\varphi$ factors uniquely through $\faktor{P_\Lambda(i)}{\ker \pi_i}\cong P_{A^{\operatorname{op}}}(i)$, i.e., there is a unique homomorphism $\overline{\varphi}: P_{A^{\operatorname{op}}}(i)\to P_{A^{\operatorname{op}}}(j)$ making the following diagram commutative.
	\begin{align*}\xymatrix@=0.3cm{P_\Lambda(i) \ar[d]_-{\pi_i} \ar[rd]^-{\varphi} \\
			P_{A^{\operatorname{op}}}(i) \ar[r]_-{\overline{\varphi}} & P_{A^{\operatorname{op}}}(j)	
	}\end{align*}
	Next, define two maps, $\Phi$ and $\Psi$, as follows.
	\begin{align*}\Phi &:\operatorname{Hom}_\Lambda(P_\Lambda(i),P_{A^{\operatorname{op}}}(j)) \to \operatorname{Hom}_{\Lambda}(P_{A^{\operatorname{op}}}(i),P_{A^{\operatorname{op}}}(j)),\quad \varphi \mapsto \overline{\varphi},  \\
	\Psi &: \operatorname{Hom}_{\Lambda}(P_{A^{\operatorname{op}}}(i),P_{A^{\operatorname{op}}}(j))\to \operatorname{Hom}_{\Lambda}(P_{\Lambda}(i),P_{A^{\operatorname{op}}}(j)),\quad \overline{\psi} \mapsto \overline{\psi}\pi_i.
\end{align*}
	From the above considerations, we know that the map $\Phi$ is well-defined. It is clear that $\Psi$ is well-defined, and that both maps are linear. Finally, we find that
	\begin{align*}
		\Phi \circ \Psi (\overline{\psi}) &= \Phi ( \overline{\psi} \pi_i)=\Phi(\psi)=\overline{\psi} \quad \textrm{and}\quad \Psi \circ \Phi (\varphi)=\Psi (\overline{\varphi})=\overline{\varphi}\pi_i=\varphi,
	\end{align*}
	showing that $\Phi$ and $\Psi$ are mutually inverse linear isomorphisms.
\end{proof}
In what follows, we will often need to consider homomorphisms between indecomposable projective $\Lambda$-modules. Such homomorphisms are (linear combinations of) homomorphisms which act by right multiplication with a certain path in the quiver of $\Lambda$. More precisely, for vertices $i$ and $j$, and a path $\xymatrixcolsep{0.3cm}\xymatrix{i\ar[r]^-p & j}$, there is a homomorphism $\rho_p: P_\Lambda(j)\to P_\Lambda(i)$ defined by $x\mapsto xp$.
\begin{proposition}\label{proposition:ext^k between Lambda-standards}
Let $P_B^\bullet$ be a minimal projective resolution of $L_B(i)$ with terms
\begin{align*}P_B^k=\medoplus\limits_{\ell=1}^n P_B(\ell)^{\oplus m_{\ell,k}},\quad k\geq 0.\end{align*}
Then, there are linear isomorphisms
\begin{align*}\operatorname{Ext}_\Lambda^k\left( \Delta_\Lambda(i),\Delta_\Lambda(j)\right)\cong \operatorname{Hom}_\Lambda(\Lambda\otimes_B P_B^k, \Delta_\Lambda(j))\cong \medoplus_{\ell=1}^n \left(e_j A e_\ell\right)^{\oplus m_{\ell,k}}.\end{align*}
\end{proposition}
\begin{proof}
By Lemma \ref{lemma:regular tensor simple iso standard & regular tensor projective iso projective,tensor preserves projective cover,lemma:tensor functors invertible on objects}, the module $\Delta_\Lambda(i)$ has a minimal projective resolution with terms
 $$\Lambda \otimes_B P_B^k\cong \medoplus_{\ell=1}^n \Lambda\otimes_B P_B(\ell)^{\medoplus m_{\ell,k}}\cong \medoplus_{\ell=1}^n P_\Lambda(\ell)^{\medoplus m_{\ell, k}}.$$
 We first show that $\operatorname{Ext}_\Lambda^k(\Delta_\Lambda(i),\Delta_\Lambda(j))\cong\operatorname{Hom}_\Lambda(\Lambda\otimes_B P_B^k,\Delta_\Lambda(j))$. Let $\partial$ be the differential on the complex $\operatorname{Hom}_\Lambda( \Lambda\otimes_B P_B^\bullet, \Delta_\Lambda(j))$, with the convention that
$$\partial^k:\operatorname{Hom}_\Lambda(\Lambda\otimes_B P_B^{k-1},\Delta_\Lambda(j)) \to \operatorname{Hom}_\Lambda(\Lambda\otimes_B P_B^{k},\Delta_\Lambda(j)).$$ 
Then, we have $\operatorname{Ext}_\Lambda^k (\Delta_\Lambda(i),\Delta_\Lambda(j))=\faktor{\ker \partial^k}{\operatorname{im}\partial^{k+1}}$. Therefore, to prove our claim, it suffices to show that $\partial$ is the zero map in each degree. The differential $\partial$ is given, in each degree, by a matrix whose entries are linear combinations of maps of the form $\blank \circ \rho_p$. Therefore, it is enough to prove that every map of the form $\blank \circ \rho_p$ is the zero map. Consider 
$$\blank\circ \rho_{p}: \operatorname{Hom}_\Lambda( P_\Lambda(\ell),\Delta_\Lambda(j))\to \operatorname{Hom}_\Lambda(P_\Lambda(\ell^\prime), \Delta_\Lambda(j)).$$
Note that $\rho_p:P_\Lambda(\ell^\prime)\to P_\Lambda(\ell)$ is given by right multiplication with a path $p$ in $B$, because it is in the image of the functor $F$.
	\begin{align*}\xymatrix@=0.3cm{
		&P_\Lambda(\ell) \ar[ld]_{\pi_\ell} \ar[dd]^{\psi}&  \ar[l]_{\rho_p} P_\Lambda(\ell^\prime) \ar[ldd]\\
		P_{A^{\operatorname{op}}}(\ell) \ar[rd]_{\overline{\psi}}\\
		&\Delta_\Lambda(j)
	}\end{align*}
	By the proof of Proposition \ref{proposition:homs from projective to standard}, any homomorphism $\psi$ factors through the projection $\pi_\ell$. For any $x\in P_\Lambda(\ell^\prime)$, we have 
	$$\pi_\ell \circ \rho_p(x)=\pi_\ell (xp)=0$$
	because $p$ is a path in $B$. We conclude that $\operatorname{im}\rho_p \subset \ker \pi_\ell$, which implies that $\psi \circ \rho_p= \overline{\psi} \circ \pi_\ell\circ  \rho_p=0.$ So, in each degree, the matrix constituting the differential $\partial$ has only zero entries. This finishes the proof of the first isomorphism in the statement of the proposition. For the second one, we apply Proposition \ref{proposition:homs from projective to standard}:
	
	\begin{align*}
	\operatorname{Ext}_\Lambda^k\left(\Delta_\Lambda(i),\Delta_\Lambda(j)\right)&\cong\operatorname{Hom}_\Lambda\left( \Lambda\otimes_B P_B^k, \Delta_\Lambda(j) \right)\cong\operatorname{Hom}_\Lambda\left( \medoplus\limits_{\ell=1}^n P_\Lambda(\ell)^{\oplus m_{\ell,k}}, \Delta_\Lambda(j)\right)\\
	&\cong\medoplus_{\ell=1}^n \operatorname{Hom}_\Lambda \left(P_\Lambda(\ell), \Delta_\Lambda(j) \right)^{\oplus m_{\ell,k}}\cong \medoplus_{\ell=1}^n \operatorname{Hom}_{\Lambda}\left(P_{A^{\operatorname{op}}}(\ell),P_{A^{\operatorname{op}}}(j)\right)^{\oplus m_{\ell,k}}\\
	&\cong \medoplus_{\ell=1}^n \left(e_\ell A^{\operatorname{op}} e_j\right)^{\oplus m_{\ell,k}}
	\cong \medoplus_{\ell=1}^n \left(e_j A e_\ell\right)^{\oplus m_{\ell,k}}. \qedhere
	\end{align*}
\end{proof}
\begin{example}\label{B in dual extension not generally regular}
	We consider again the algebra $\Lambda=\mathcal{A}(B,B^{\operatorname{op}})$ where $B=A=K( \xymatrixcolsep{0.3cm}\xymatrix{1\ar[r]^-\alpha & 2\ar[r]^-\beta & 3})$. We wish to compare the spaces $\operatorname{Ext}_B^1(L_B(1),L_B(3))$ and $\operatorname{Ext}_\Lambda^1(\Delta_\Lambda(1),\Delta_\Lambda(3)).$Recall that $\Lambda$ is given by the quiver
	$$\xymatrixcolsep{0.3cm}\xymatrix{1\ar@/^.5pc/[r]^-{\alpha} & 2 \ar@/^.5pc/[r]^-{\beta}\ar@/^.5pc/[l]^-{\alpha^\prime} &3 \ar@/^.5pc/[l]^-{\beta\prime}}$$
	subject to the relations $\alpha \alpha^\prime=0$ and $\beta\beta^\prime=0$. We immediately see that
	$$\dim \operatorname{Ext}_B^1(L_B(1), L_B(3))=0,$$
	since this dimension coincides with the number of arrows $\xymatrixcolsep{0.3cm}\xymatrix{1 \ar[r]^-\gamma & 3}$ in the quiver of $B$, and there are zero such arrows. Since we have a minimal projective resolution 
	$$\xymatrixcolsep{0.3cm}\xymatrix{0\ar[r] & P_B(2) \ar[r] & P_B(1) \ar[r] & L_B(1),}$$ it follows from Proposition \ref{proposition:ext^k between Lambda-standards} and Lemma \ref{lemma:regular tensor simple iso standard & regular tensor projective iso projective,tensor preserves projective cover,lemma:tensor functors invertible on objects}, part (3), that we have
	$$\dim \operatorname{Ext}_\Lambda^1(\Delta_\Lambda(1),\Delta_\Lambda(3))\cong \operatorname{Hom}_\Lambda(P_\Lambda(2), \Delta_\Lambda(3)),$$ and this space contains the map $f$ given by right multiplication by the arrow $\beta^\prime$.
	This means that $\dim \operatorname{Ext}_\Lambda(\Delta_\Lambda(1),\Delta_\Lambda(3)) \geq 1$, so that $B$ is not regular.
\end{example}
Proposition \ref{proposition:ext^k between Lambda-standards} sheds some further light on how $B\subset \Lambda$ fails to be a regular exact subalgebra. Under the assumption of the proposition, one can check that $\operatorname{Ext}_B^k(L_B(i),L_B(j))\cong \operatorname{Hom}_B(P_B^k, L_B(j))$, meaning we get an extension for each copy of $P_B(j)$ appearing in $P_B^k$, because
$$\dim \operatorname{Hom}_B(P_B(\ell), L_B(j))=\begin{cases}
	1, & \textrm{ if } \ell=j,\\
	0& \textrm{ otherwise}.
\end{cases}$$
Similarly, for the standard modules over $\Lambda$, we saw that $\operatorname{Ext}_\Lambda^k(\Delta_\Lambda(i),\Delta_\Lambda(j))\cong \operatorname{Hom}_\Lambda(\Lambda\otimes_B P_B^k, \Delta_\Lambda(j))$. This space decomposes into a direct sum of spaces of the form $\operatorname{Hom}_\Lambda(P_\Lambda(\ell),\Delta_\Lambda(j))$. Of course, if $\ell=j$, this space contains the projection $\pi_j: P_\Lambda(j)\to \Delta_\Lambda(j)$, which is the image of the projection $P_B(j)\to L_B(j)$ under the functor $\Lambda\otimes_B \blank$. However, in Proposition \ref{proposition:homs from projective to standard}, we saw that the spaces $\operatorname{Hom}_\Lambda(P_\Lambda(\ell),\Delta_\Lambda(j))$ are in general not zero for $\ell\neq j$, yielding additional extensions which are not contained in the image of the functor $\Lambda\otimes_B \blank$.
\section{The $\operatorname{Ext}$-algebra of standard modules}
Put $\Delta=\Delta_\Lambda(1)\oplus\dots \oplus \Delta_\Lambda(n)$ and $\mathbb{L}=L_B(1)\oplus\dots\oplus L_B(n).$
The space $\operatorname{Ext}_\Lambda^\ast(\Delta,\Delta)=\medoplus_{k\geq0} \operatorname{Ext}^k_\Lambda(\Delta,\Delta)$ has a natural structure of a graded algebra (as does the space $\operatorname{Ext}_B^\ast(\mathbb{L},\mathbb{L})$), with the multiplication given by Yoneda product. This section is devoted to its description. Let $P_\Lambda^\bullet$ and $Q_\Lambda^\bullet$ be minimal projective resolutions of $\Delta_\Lambda(i)$ and $\Delta_\Lambda(j)$, respectively. Then, $\operatorname{Ext}_\Lambda^k\left(\Delta_\Lambda(i),\Delta_\Lambda(j)\right)\cong \operatorname{Hom}_{\mathcal{K}(\Lambda\operatorname{-proj})}\left(P_\Lambda^\bullet, Q_\Lambda^\bullet[k]\right)$, where $[\blank ]$ denotes the shift functor on complexes. Under this isomorphism, the Yoneda product corresponds to composition of (equivalence classes of) chain maps. We wish to find chain maps corresponding to basis elements of $\operatorname{Ext}_\Lambda^k(\Delta_\Lambda(i),\Delta_\Lambda(j))$ under this isomorphism. Recall that $\operatorname{Ext}_\Lambda^k\left( \Delta_\Lambda(i),\Delta_\Lambda(j)\right)\cong \operatorname{Hom}_\Lambda(P_\Lambda^k, \Delta_\Lambda(j))$ by Proposition \ref{proposition:ext^k between Lambda-standards}, so a basis element of $\operatorname{Ext}_\Lambda^k(\Delta_\Lambda(i),\Delta_\Lambda(j))$ is represented by $\varphi\in  \operatorname{Hom}_\Lambda(P_\Lambda^k, \Delta_\Lambda(j))$.

\begin{align*}\xymatrix{
	P_\Lambda^\bullet:\dots \ar[r] & P_\Lambda^{k+1} \ar@{-->}[d]^-{\varphi_{k+1}}\ar[r]^-{\partial_{k+1}} & P_\Lambda^{k} \ar[rd]^{\varphi}\ar[r]^{\partial_k} \ar@{-->}[d]^{\varphi_k}& P_\Lambda^{k-1}  \ar[r]& \dots \\
	Q_\Lambda^\bullet[k]: \dots \ar[r] & Q_\Lambda^1 \ar[r] & P_\Lambda(j) \ar@{->>}[r]_{\pi_j} & \Delta_\Lambda(j) \ar[r]& \dots
}\end{align*}
Such a chain map should have components $\varphi_k, \varphi_{k+1},\dots$ as indicated above. Additionally, the component $\varphi_k$ should lift the homomorphism $\varphi$, that is, $\pi_j\circ  \varphi_k=\varphi$. In the above picture, the maps are matrices whose entries are homomorphisms between indecomposable projective $\Lambda$-modules, or in the case of the map $\varphi$, homomorphisms $P_\Lambda(x)\to \Delta_\Lambda(j)$. Assume that the projective module $P_\Lambda(j)$ does not appear in the direct sum decomposition of the module $P_\Lambda^k$. By Proposition \ref{proposition:homs from projective to standard}, each entry of the matrix consituting the homomorphism $\varphi: P_\Lambda^k\to \Delta_\Lambda(j)$ factors as $\rho_{\gamma^\prime} \circ \pi_x$, where $x$ is such that $P_\Lambda(x)$ is a direct summand of $P_\Lambda^k$ and ${\gamma^\prime}$ is a linear combination of paths in the quiver of $A^{\operatorname{op}}$. 
$$\xymatrix{
	P_\Lambda(x) \ar[d]_-{\rho_{\gamma^\prime}} \ar[rd]^-{\rho_{\gamma^\prime} \circ \pi_x} \\
	P_\Lambda(j) \ar[r]_-{\pi_j} & \Delta_\Lambda(j)
}$$
Next, we claim that the above triangle commutes. Recall that the module $P_\Lambda(x)$ has a basis made up of elements of the form $q^\prime p$ where $p$ is a path in $B$, starting in $x$, and $q^\prime$ is a path in $A^{\operatorname{op}}$. If $p=e_x$, we have 
$$\rho_{\gamma^\prime} \circ \pi_x (q^\prime p)=\rho_{\gamma^\prime} \circ \pi_x(q^\prime)= \rho_{\gamma^\prime}(q^\prime)=q^\prime {\gamma^\prime} =\pi_j \circ \rho_{\gamma^\prime} (q^\prime)=\pi_j\circ \rho_{\gamma^\prime}(q^\prime p)$$
and, if $p\neq e_x$, we  have
$$\rho_{\gamma^\prime} \circ \pi_x(q^\prime p)=0=\pi_j (q^\prime p {\gamma^\prime} )=\pi_j \circ \rho_{\gamma^\prime}( q^\prime p),$$
since $p{\gamma^\prime}=0$ according to the dual extension relation, which proves the claim. Then, taking the map $\varphi_k$ to be the matrix having an entry $\rho_{\gamma^\prime}$ whenever the matrix describing $\varphi$ has an entry $\rho_{\gamma^\prime}\circ \pi_x$, we see that the following diagram commutes.

$$\xymatrix{
	P_\Lambda^k \ar[d]_-{\varphi_k} \ar[rd]^-{\varphi} \\
	P_\Lambda(j) \ar[r]_-{\pi_j} & \Delta_\Lambda(j)
}$$
Next, note that the differentials on the complexes $P_\Lambda^\bullet$ and $Q_\Lambda^\bullet$ are matrices whose entries are linear combinations of maps $\rho_\alpha$, where $\alpha$ is a path in $B$. If $\beta^\prime$ is a path in $A^{\operatorname{op}}$, we have $\rho_{\beta^\prime}\circ \rho_\alpha=\rho_{\alpha \beta^\prime}=0$ since $\alpha \beta^\prime=0$ according to the dual extension relation. Returning to the picture
\begin{align*}\xymatrix{
		P_\Lambda^\bullet:\dots \ar[r] & P_\Lambda^{k+1} \ar@{-->}[d]^-{\varphi_{k+1}}\ar[r]^-{\partial_{k+1}} & P_\Lambda^{k} \ar[rd]^{\varphi}\ar[r]^{\partial_k} \ar@{-->}[d]^{\varphi_k}& P_\Lambda^{k-1}  \ar[r]& \dots \\
		Q_\Lambda^\bullet[k]: \dots \ar[r] & Q_\Lambda^1 \ar[r] & P_\Lambda(j) \ar@{->>}[r]_{\pi_j} & \Delta_\Lambda(j) \ar[r]& \dots
},\end{align*}
 our observations now imply that we may extend $\varphi_k$ to a chain map by putting $\varphi_{k+1}=\varphi_{k+2}=\dots=0.$ First, we showed that $\pi_j\circ \varphi_k=\varphi$. Then, we checked that $\varphi_k\circ \partial_{k+1}=0$. All other squares commute trivially, meaning we have found a chain map representative of the extension given by $\varphi$.

Note that the above contruction does not work if the projective module $P_\Lambda(j)$ appears in the direct sum decomposition of $P_\Lambda^k$. If it does, the space $\operatorname{Hom}_\Lambda(P_\Lambda^k,\Delta_\Lambda(j))$ contains the projection $\pi_j:P_\Lambda(j)\to\Delta_\Lambda(j)$, which may be lifted to the identity homomorphism on $P_\Lambda(j)$. Then, the collection of maps $\varphi_k,\varphi_{k+1},\dots$ given above is no longer a chain map.

Of particular importance is the special case when $k=0$, corresponding to homomorphisms, rather than proper extensions, from $\Delta_\Lambda(i)$ to $\Delta_\Lambda(j)$. Then, the first part of our picture looks like
$$\xymatrix{
	\dots  \ar[r]& P_\Lambda(i)  \ar@{-->}[d]^-{\varphi_0} \ar[dr]^-{\varphi} \ar[r] & \Delta_\Lambda(i) \\
	\dots \ar[r] & P_\Lambda(j) \ar[r] & \Delta_\Lambda(j),
}$$
which ensures that our construction goes through.
\begin{proposition}\label{proposition:m_2 in ext-algebra}
Let $i,j$ and $\ell$ be vertices such that $i\neq j$, and assume $k\geq 1$. Then, the multiplication
	\begin{align*}m_2:\operatorname{Ext}^k_\Lambda(\Delta_\Lambda(j),\Delta_\Lambda(\ell)) \times \operatorname{Hom}_\Lambda(\Delta_\Lambda(i),\Delta_\Lambda(j)) \to \operatorname{Ext}_\Lambda^{k}(\Delta_\Lambda(i),\Delta_\Lambda(\ell))\end{align*}
is the zero map.
\end{proposition}
\begin{proof}
	Let $P_\Lambda^\bullet, Q_\Lambda^\bullet$ and $R_\Lambda^\bullet$ be minimal projective resolutions of $\Delta_\Lambda(i),\Delta_\Lambda(j)$ and $\Delta_\Lambda(\ell)$, respectively. The fact that $i$ and $j$ are distinct ensures that, if $\varphi\in \operatorname{Hom}_\Lambda(\Delta_\Lambda(i),\Delta_\Lambda(j))$ is a homomorphism, our construction of the chain map representing $\varphi$, in the previous discussion, goes through. Consider the following diagram.
	
	$$\xymatrix{
	P_\Lambda^\bullet: \dots & \dots & \dots  & \dots\ar[r] & P_\Lambda^1\ar[r] \ar@{-->}[d]^-{0} & \ar@{-->}[d]^-{\varphi_0} P_\Lambda^0\ar[dr]^-{\varphi } \ar[r] & \Delta_\Lambda(i) \\
	Q_\Lambda^\bullet: \dots \ar[r] & Q_\Lambda^{k+1} \ar[r] \ar@{-->}[rrrd]_-{\psi_{k+1}} & Q_\Lambda^k \ar[r]\ar@/^1ex/[rrrrd]^->>>>>>>>>{\psi} \ar@{-->}[rrrd]^-{\psi_k}& \dots\ar[r] & Q_\Lambda^1\ar[r] & Q_\Lambda^0 \ar[r] & \Delta_\Lambda(j) \\
	R_\Lambda^\bullet: \dots &\dots &\dots & \dots\ar[r]&R_\Lambda^1 \ar[r]& R_\Lambda^0 \ar[r] & \Delta_\Lambda(\ell)
	}$$
	 Since the only nonzero component of the chain map representing $\varphi$ has codomain $Q_\Lambda^0$ and $k\geq 1$, the statement follows.
\end{proof}
\begin{theorem}\label{theorem:ext-algebra of standards isomorphic to dual extension of ext-algebra of simples and B}There are isomorphisms of graded algebras
	\begin{align*}\operatorname{Ext}_\Lambda^\ast\left(\Delta,\Delta\right) \cong \mathcal{A}\left(\operatorname{Ext}_B^\ast(\mathbb{L},\mathbb{L}), \operatorname{Hom}_\Lambda(\Delta,\Delta)\right)\cong  \mathcal{A}\left(\operatorname{Ext}_B^\ast(\mathbb{L},\mathbb{L}), A\right).\end{align*}
\end{theorem}
\begin{proof}
	The algebras $\operatorname{Ext}_\Lambda^\ast(\Delta,\Delta)$ and $\operatorname{Ext}_B^\ast(\mathbb{L},\mathbb{L})$ are naturally graded by the degree of extensions. We extend these to gradings on $\mathcal{A}\left(\operatorname{Ext}_B^\ast(\mathbb{L},\mathbb{L}), \operatorname{Hom}_\Lambda(\Delta,\Delta)\right)$ and on $\mathcal{A}\left(\operatorname{Ext}_B^\ast(\mathbb{L},\mathbb{L}), A\right)$ by letting elements of $\operatorname{Hom}_\Lambda(\Delta,\Delta)$ and of $A$ be homogeneous of degree 0. We first prove that $\operatorname{Ext}_\Lambda^\ast\left(\Delta,\Delta\right) \cong \mathcal{A}\left(\operatorname{Ext}_B^\ast(\mathbb{L},\mathbb{L}), \operatorname{Hom}_\Lambda(\Delta,\Delta)\right)$.
	
	Suppose the simple $B$-modules $L_B(i),L_B(j)$ and $L_B(\ell)$ have minimal projective resolutions $P_B^\bullet,Q_B^\bullet$ and $R_B^\bullet$, respectively. Let $P_\Lambda^\bullet,Q_\Lambda^\bullet$ and  $R_\Lambda^\bullet$ be the induced projective resolutions of $\Delta_\Lambda(i),\Delta_\Lambda(j)$ and $\Delta_\Lambda(\ell)$, respectively. By Proposition \ref{proposition:ext^k between Lambda-standards}, there holds $\operatorname{Ext}_\Lambda^k(\Delta_\Lambda(i),\Delta_\Lambda(j))\cong\operatorname{Hom}_\Lambda(P_\Lambda^k,\Delta_\Lambda(j)).$
	It is easy to check that, similarly, there holds $\operatorname{Ext}^k_B(L_B(i),L_B(j))\cong\operatorname{Hom}_B(P_B^k, L_B(j)).$ Lemma \ref{lemma:regular tensor simple iso standard & regular tensor projective iso projective,tensor preserves projective cover,lemma:tensor functors invertible on objects} implies that the functor $F$ yields a (linear) map \begin{align*}\Lambda\otimes_B\blank: \operatorname{Hom}_B(P_B^k,L_B(j)) \to \operatorname{Hom}_\Lambda(P_\Lambda^k,\Delta_B(j)).\end{align*}
	As the functor $F$ is faithful, the map is injective. Considering the diagram
	
	\begin{align*}\xymatrix{
		\operatorname{Ext}^k_B(L_B(i),L_B(j)) \ar[r]^-{\Lambda\otimes_B\blank} \ar@{}[d]|*=0[@]{\cong}& \operatorname{Ext}^k_\Lambda(\Delta_\Lambda(i),\Delta_B(j)) \ar@{}[d]|*=0[@]{\cong}\\
		\operatorname{Hom}_{\mathcal{K}(B)}\left(P_B^\bullet,Q_B^\bullet[k]\right) \ar@{-->}[r]& \operatorname{Hom}_{\mathcal{K}(\Lambda)}\left(P_\Lambda^\bullet,Q_\Lambda^\bullet[k]\right)
	}\end{align*}
	we see that it may be made to commute by letting the dashed arrow represent the map which takes a chain map $(f)_i$ to the chain map $(\operatorname{id}_\Lambda \otimes f)_i.$ For $f\in \operatorname{Hom}_{\mathcal{K}(B)}(P_B^\bullet,Q_B^\bullet[k])$ and $g\in \operatorname{Hom}_{\mathcal{K}(B)}(Q_B^\bullet,R_B^\bullet[k^\prime])$ the chain map $(gf)_i\in \operatorname{Hom}_{\mathcal{K}(B)}(P_B^\bullet,R_B^\bullet[k+k^\prime])$
	has components $g_{i+k}f_i:P_B^i\to R_B^{i+k+k^\prime}.$ We check that $\Lambda\otimes_B\blank$ is compatible with composition of chain maps and thus extends to an injective homomorphism of (graded) algebras $\Lambda\otimes_B\blank: \operatorname{Ext}_B^\ast(\mathbb{L},\mathbb{L})\hookrightarrow \operatorname{Ext}_\Lambda^\ast(\Delta,\Delta)$.
	\begin{align*}
	(\operatorname{id}_\Lambda \otimes g_{i+k}f_i )(x\otimes y)&=x\otimes g_{i+k}f_i (y)=(\operatorname{id}_\Lambda \otimes g_{i+k})(x\otimes f_i(y))=(\operatorname{id}_\Lambda \otimes g_{i+k})(\operatorname{id}_\Lambda \otimes f_{i})(x\otimes y).
	\end{align*}
	
	Next, if $\varepsilon\in \Lambda\otimes_B \operatorname{Ext}_B^\ast(\mathbb{L},\mathbb{L})$ is such that $\deg \varepsilon\geq 1$ and $\varphi \in \operatorname{Hom}_\Lambda(\Delta_\Lambda(i),\Delta_\Lambda(j))\subset \operatorname{Ext}_\Lambda^\ast(\Delta,\Delta)$ with $i\neq j$, then $m_2(\varepsilon,\varphi)=0$ by Proposition \ref{proposition:m_2 in ext-algebra}. This fact, together with the existence of the embeddings of algebras $\operatorname{Ext}_B^\ast(\mathbb{L},\mathbb{L}) \hookrightarrow \operatorname{Ext}_\Lambda^\ast(\Delta,\Delta)$ and $\operatorname{Hom}_\Lambda(\Delta,\Delta) \hookrightarrow \operatorname{Ext}_\Lambda^\ast(\Delta,\Delta)$, implies that there is a homomorphism of graded algebras
	$$\Phi: \mathcal{A}\left(\operatorname{Ext}_B^\ast(\mathbb{L},\mathbb{L}), \operatorname{Hom}_\Lambda(\Delta,\Delta)\right)\to\operatorname{Ext}_\Lambda^\ast(\Delta,\Delta).$$
	
	We claim that $\Phi$ is surjective. Suppose that the resolution $P_B^\bullet \to L_B(i)$ has terms $P_B^k=\medoplus_{t=1}^n P_B(t)^{\oplus m_{t,k}}$. Then, the resolution $P_\Lambda^\bullet \to \Delta_B(i)$ has terms $P_\Lambda^k=\medoplus_{t=1}^n P_\Lambda(t)^{\oplus m_{t,k}}$. Fix an extension
	$$\alpha \in \operatorname{Ext}_\Lambda^k(\Delta_\Lambda(i),\Delta_\Lambda(j))\cong \operatorname{Hom}_\Lambda(P_\Lambda^k,\Delta_\Lambda(j)\cong\medoplus_{t=1}^n \operatorname{Hom}_\Lambda(P_\Lambda(t),\Delta_\Lambda(j))^{\oplus m_{t,k}},$$
 where we use Proposition \ref{proposition:ext^k between Lambda-standards}, and additivity of the Hom-functor, to obtain the chain of isomorphisms. The extension $\alpha$ is represented by a matrix, whose entries are maps $\alpha_t^1,\dots, \alpha_t^{m_{t,k}}\in \operatorname{Hom}_\Lambda(P_\Lambda(t),\Delta_\Lambda(j))$. Here, $\alpha_t^r$ is the entry correspondng to the $r$th copy of $P_\Lambda(t)$ occuring in $P_\Lambda^k$. Similarly, we have
	$$\operatorname{Ext}_B^k(L_B(i),L_B(j))=\operatorname{Hom}_B(P_B^k,L_B(j))\cong \medoplus_{t=1}^n\operatorname{Hom}_B(P_B(t),L_B(j))^{m_{t,k}}.$$
	The spaces on $\operatorname{Hom}_B(P_B(t),L_B(j))$ are zero unless $t=j$, in which case they contain exactly the projection $P_B(j)\to L_B(j)$ (up to a scalar). This means that the image of $\operatorname{Hom}_B(P_B(t),L_B(j))$ under the functor $F$ is spanned by the projection $\pi_j: P_\Lambda(j)\to \Delta_\Lambda(j)$ if $t=j$, and 0 otherwise. Now we apply Proposition \ref{proposition:homs from projective to standard} to see that any $\alpha_t^r\in \operatorname{Hom}_\Lambda(P_\Lambda(t),\Delta_\Lambda(j))$ may be written as $\alpha_t^r=\overline{\alpha_t^r}\pi_t$ where $\pi_t:P_\Lambda(t)\to \Delta_\Lambda(t)$ is the natural projection and $\overline{\alpha_t^r}\in \operatorname{Hom}_\Lambda(\Delta_\Lambda(t),\Delta_\Lambda(j))$. This shows that the image under the functor $F$ of $\operatorname{Ext}_B^\ast(\mathbb{L},\mathbb{L})$, together with $\operatorname{Hom}_\Lambda(\Delta,\Delta)$, is enough to generate $\operatorname{Ext}_\Lambda^\ast(\Delta,\Delta)$, so $\Phi$ is surjective. To finish the proof, we count dimensions. Using Propositions \ref{proposition:homs from projective to standard} and \ref{proposition:ext^k between Lambda-standards}, we have
	\begin{align*}\dim \operatorname{Ext}_\Lambda^k(\Delta_\Lambda(i),\Delta_\Lambda(j))&=\dim \operatorname{Hom}_\Lambda(P_\Lambda^k, \Delta_\Lambda(j))=\dim \operatorname{Hom}_\Lambda \left( \medoplus_{\ell=1}^n P_\Lambda(\ell)^{\oplus m_{\ell,k}}, \Delta_\Lambda(j)\right)=\\
		&=\dim \medoplus_{\ell=1}^n \operatorname{Hom}_\Lambda(P_\Lambda(\ell),\Delta_\Lambda(j))^{\oplus m_{\ell,k}}=\sum_{\ell=1}^n m_{\ell,k}\dim \operatorname{Hom}_\Lambda(P_\Lambda(\ell),\Delta_\Lambda(j)).
	\end{align*}
We compare this to the degree $k$ part of $e_j \mathcal{A}(\operatorname{Ext}^\ast_B(\mathbb{L},\mathbb{L}),\operatorname{Hom}_\Lambda(\Delta,\Delta))e_i.$ Such an element may only be obtained by multiplying an extension $\varepsilon_k \in \Lambda\otimes_B \operatorname{Ext}^\ast_B(\mathbb{L},\mathbb{L})$, with $\deg \varepsilon_k=k$, with an element $\varphi\in \operatorname{Hom}_\Lambda(\Delta,\Delta)$, on the left. That is, an element of degree $k$ should have the form $\varphi \varepsilon_k$. For this composition to be nonzero, we must have $\varepsilon_k\in \Lambda\otimes_B\operatorname{Ext}_B^k(L_B(i),L_B(\ell))$ and $\varphi \in \operatorname{Hom}_\Lambda(\Delta_\Lambda(\ell),\Delta_\Lambda(j))$ for some $\ell$. For the extensions between the simple modules, we have
\begin{align*}
	\dim \operatorname{Ext}_B^k (L_B(i),L_B(\ell))&= \dim \operatorname{Hom}_B (P_B^k, L_B(\ell))=\dim \operatorname{Hom}_B \left( \medoplus_{t=1}^n P_B(t)^{\oplus m_{t,k}}, L_B(\ell) \right)\\
	&=\sum_{t=1}^n m_{t,k} \dim \operatorname{Hom}_B(P_B(t),L_B(\ell))= m_{\ell,k}\cdot \dim \operatorname{Hom}_B ( P_B(\ell), L_B(\ell))=m_{\ell,k}
\end{align*}
since $\dim \operatorname{Hom}_B (P_B(t),L_B(\ell))=1$ if $t=\ell$ and zero otherwise.
Summing the possibilities over all $\ell$, we get that the degree $k$ part of $e_j \mathcal{A}(\operatorname{Ext}_B(\mathbb{L},\mathbb{L}),\operatorname{Hom}_\Lambda(\Delta,\Delta))e_i$ has dimension
	\begin{align*} \sum_{\ell=1}^n \dim \operatorname{Ext}_B^k(L_B(i),L_B(\ell)) \cdot \dim \operatorname{Hom}_\Lambda(\Delta_\Lambda(\ell),\Delta_\Lambda(j))&=\sum_{\ell=1}^n m_{\ell,k} \cdot \dim \operatorname{Hom}_\Lambda(\Delta_\Lambda(\ell),\Delta_\Lambda(j))
	.\end{align*}
	
	From this, it follows that $\dim \mathcal{A}\left(\operatorname{Ext}_B^\ast(\mathbb{L},\mathbb{L}),\operatorname{Hom}_\Lambda(\Delta,\Delta)\right) = \dim \operatorname{Ext}_\Lambda^\ast(\Delta,\Delta)$
	so that $\Phi$ is a surjective linear map between vector spaces of the same dimension, hence an isomorphism. This proves the first isomorphism in the statement. For the second, we note that as $\Lambda$-modules, we have
	\begin{align*}\Delta=\Delta_\Lambda(1)\oplus \dots, \oplus \Delta_\Lambda(n)\cong P_{A^{\operatorname{op}}}(1)\oplus \dots, \oplus P_{A^{\operatorname{op}}}(n)\cong A^{\operatorname{op}}e_1\oplus\dots\oplus A^{\operatorname{op}}e_n\cong A^{\operatorname{op}}\end{align*}
	and since $\operatorname{End}_\Lambda(A^{\operatorname{op}})\cong A$ as algebras, we get
	$$\operatorname{Ext}_\Lambda^\ast(\Delta,\Delta)\cong \mathcal{A}(\operatorname{Ext}_B^\ast(\mathbb{L},\mathbb{L}), \operatorname{Hom}_\Lambda(\Delta,\Delta))\cong \mathcal{A}(\operatorname{Ext}_B^\ast(\mathbb{L},\mathbb{L}),A). \hfill \qedhere$$
\end{proof}
\section{Koszulity}
In this section, we investigate some of the properties of $\Lambda$ as a graded algebra. Any directed algebra admits a $\mathbb{Z}$-grading by path length, and these gradings extend to a grading on $\Lambda=\mathcal{A}(B,A^{\operatorname{op}})$. Recall that a graded $\Lambda$-module $M$ is said to be \emph{generated in degree} $i$ if $\Lambda M_i=M$. If $M$ has a projective resolution consisting of graded $\Lambda$-modules $\xymatrix@=0.3cm{\dots\ar[r]&P_1\ar[r] &P_0\ar[r]& M}$
such that $P_i$ is generated in degree $i$ for all $i\geq 0$, we say that $M$ has a \emph{linear resolution}. The algebra $\Lambda$ is said to be \emph{Koszul} if every simple module $L_\Lambda(i)$ has a linear resolution. We remark that in these notions, $\Lambda$ may in principle be any graded algebra. For further details on graded algebras and modules, we refer to \cite{GORDONGREEN}. In \cite{Lixu}, the authors prove the following theorem, connecting the koszulity of $\mathcal{A}(B,A^{\operatorname{op}})$ to the koszulity of $B$ and of $A$.

\begin{theorem}\label{theorem:Lambda koszul iff A and B koszul} \cite[Proposition~3.5]{Lixu}
	The algebra $\Lambda=\mathcal{A}(B,A^{\operatorname{op}})$ is Koszul if and only if both $B$ and $A$ are Koszul.
\end{theorem}
This motivates the question of whether or not there is some corresponding statement concerning the existence of linear resolutions of the standard modules $\Delta_\Lambda(i)$. 
\begin{definition} Let $\Lambda$ be a graded quasi-hereditary algebra. Then, $\Lambda$ is said to be
	\begin{enumerate}[label=$(\roman*)$]
		\item \emph{left standard Koszul} if the left standard modules $\Delta_\Lambda(i)\in \Lambda\operatorname{-mod}$ have linear resolutions, for all $i\in\{1,\dots, n\}$.
		\item \emph{right standard Koszul} if the right standard modules $\Delta_\Lambda^{\operatorname{op}}(i) \in \Lambda^{\operatorname{op}}\operatorname{-mod}$ have linear resolutions, for all $i\in \{1,\dots, n\}$.
	\end{enumerate}
\end{definition}
\begin{theorem}\label{theorem: lambda left standard koszul iff b koszul}
	The algebra $\Lambda$ is left standard Koszul if and only if $B$ is Koszul, and right standard Koszul if and only if $A$ is Koszul.
\end{theorem}
\begin{proof}
	By Lemma \ref{lemma:regular tensor simple iso standard & regular tensor projective iso projective,tensor preserves projective cover,lemma:tensor functors invertible on objects}, $P_B^\bullet \to L_B(i)$ is a minimal projective resolution if and only if $F(P_B^\bullet)=P_\Lambda^\bullet \to \Delta_\Lambda(i)$ is a minimal projective resolution. Since the degree of the idempotents $e_j$ is zero in both the grading on $B$ and the grading on $\Lambda$, the $k$th term of $P_B^\bullet$, the module $P_B^k$, is generated in degree $k$ if and only if the $k$th term of $P_\Lambda^\bullet$, the module $P_\Lambda^k$, is generated in degree $k$. Therefore, $L_B(i)$ has a linear resolution if and only if $\Delta_\Lambda(i)$ has a linear resolution, so $B$ is Koszul if and only if $\Lambda$ is left standard Koszul.
	
	For the second statement, note that $\Lambda$ is right standard Koszul if and only if $\Lambda^{\operatorname{op}}$ is left standard Koszul. But, $\Lambda^{\operatorname{op}}={\mathcal{A}(B,A^{\operatorname{op}})}^{\operatorname{op}}=\mathcal{A}(A, B^{\operatorname{op}})$, which is left standard Koszul if and only if $A$ is Koszul, by the first part of the proof.
\end{proof}
Let $C$ be some Koszul algebra and put $\mathbb{L}=L_C(1)\oplus\dots\oplus L_C(n)$. Then, it is well-known that the condition of $\mathbb{L}$ having a linear resolution is equivalent to the internal and homological gradings on $\operatorname{Ext}^\ast_C(\mathbb{L},\mathbb{L})$ coinciding. This phenomenon does not generalize to the case of standard koszulity with the current grading on the dual extension algebra.
\begin{example}
Let $Q$ be the quiver $\xymatrix@=0.3cm{1\ar[r]^-\alpha&2}$, put $B=KQ$, and consider $\Lambda=\mathcal{A}(B,B^{\operatorname{op}})$. Then $\Lambda$ is given by the quiver
\begin{align*}\xymatrix@=0.3cm{
		1\ar@/^0.5pc/[r]^-{\alpha} & 2 \ar@/^0.5pc/[l]^-{\alpha^\prime}}\end{align*}
with the relation $\alpha\alpha^\prime=0$. The standard modules are $$\Delta_\Lambda(1)\cong L_\Lambda(1)\quad\textrm{and}\quad 
\Delta_\Lambda(2):\vcenter{\xymatrix@=0.3cm{2 \ar[d]^-{\alpha^\prime} \\ 1}}.$$ Putting $\Delta=\Delta_\Lambda(1)\oplus \Delta_\Lambda(2)$, the space $\operatorname{Hom}_\Lambda(\Delta,\Delta)$ contains the inclusion $\Delta_\Lambda(1)\hookrightarrow \Delta_\Lambda(2)$, which is of homological degree 0 but internal degree 1.
\end{example}
\subsection{Alternate grading on the dual extension algebra}
To remedy the situation in the previous example, we define a new grading on the dual extension algebra $\Lambda=\mathcal{A}(B,A^{\operatorname{op}})$. We still grade $B$ by path length, but put $\deg \alpha^\prime=0$ for all arrows $\alpha^\prime$ coming from the quiver of $A^{\operatorname{op}}$. Note that the idempotents $e_j$ are still homogeneous of degree zero. Since the proof of Theorem \ref{theorem: lambda left standard koszul iff b koszul} relies only on the degree of the idempotents being zero, its conclusion remains true under the new grading on $\Lambda$.

\begin{lemma}\label{lemma:homological and internal gradings on ext(delta,delta) coincide}
	Let $\Lambda=\mathcal{A}(B,A^{\operatorname{op}})$ be left standard Koszul and graded as above. Then, the homological and internal gradings on $\operatorname{Ext}_\Lambda^\ast(\Delta,\Delta)$ coincide.
\end{lemma}
\begin{proof}
	Let $P_\Lambda^\bullet$ be a minimal projective resolution of $\Delta_\Lambda(i)$. In the proof of Proposition \ref{proposition:ext^k between Lambda-standards}, we saw that $\operatorname{Ext}^k_\Lambda(\Delta(i),\Delta(j))\cong \operatorname{Hom}_\Lambda(P_\Lambda^k,\Delta_\Lambda(j))$, where $P_\Lambda^k$ is the $k$th term of $P_\Lambda^\bullet.$ The space $\operatorname{Hom}_\Lambda(P_\Lambda^k,\Delta_\Lambda(j))$ is isomorphic to a direct sum of spaces of the form $\operatorname{Hom}_\Lambda(P_\Lambda(\ell),\Delta_\Lambda(j)).$ By Proposition \ref{proposition:homs from projective to standard}, there holds $$\operatorname{Hom}_\Lambda(P_\Lambda(\ell),\Delta_\Lambda(j))\cong \operatorname{Hom}_{A^{\operatorname{op}}}(P_{A^{\operatorname{op}}}(\ell),P_{A^{\operatorname{op}}}(j)).$$
	
	\begin{enumerate}[label=$(\roman*)$]
		\item If $\ell=j$, the space $\operatorname{Hom}_{A^{\operatorname{op}}}(P_{A^{\operatorname{op}}}(\ell),P_{A^{\operatorname{op}}}(\ell))$ equals the span of the identity homomorphism on $P_{A^{\operatorname{op}}}(\ell)$.
		\item If $\ell\neq j$, any homomorphism $\varphi\in \operatorname{Hom}_{A^{\operatorname{op}}}(P_{A^{\operatorname{op}}}(\ell),P_{A^{\operatorname{op}}}(j))$ is of the form $\rho_{q^\prime}$, where $q^\prime$ is a linear combination of paths in $A^{\operatorname{op}}$.
	\end{enumerate}
In either case, such a map is homogeneous of degree 0, but when accounting for present degree shifts, we see that the internal degree is $k$.
\end{proof}
\section{$A_\infty$-structure on $\operatorname{Ext}_\Lambda^\ast(\Delta,\Delta)$}
\subsection{Background and the Koszul case}
The goal of this section is to describe completely an $A_\infty$-structure on $\operatorname{Ext}_\Lambda^\ast(\Delta,\Delta)$, given a certain $A_\infty$-structure on $\operatorname{Ext}_B^\ast(\mathbb{L},\mathbb{L})$. We describe the construction given in \cite{LuPalmieriWuZhang}, which is a special case of the construction in \cite{Merkulov}, which in turn is a special case of Kadeishvili's original construction in \cite{Kadeishvili}.

Let $M$ be a $\Lambda$-module with minimal projective resolution $P^\bullet\to M$. Form the algebra $\mathcal{D}\coloneqq\operatorname{End}_\Lambda (P^\bullet, P^\bullet)$. Note that the elements of $\mathcal{D}$ are not necessarily chain maps.

We may endow $\mathcal{D}$ with the structure of a dg-algebra by defining a differential on homogeneous maps by
\begin{align*}\partial(f)\coloneqq d^P f - (-1)^{|f|} fd^P,\end{align*}
where $d^P$ denotes the differential on $P^\bullet$, and $|f|$ denotes the (homological) degree of $f$. We have the following easy observation.

\begin{lemma}\label{lemma:merkulov's construction}
Let $f\in \mathcal{D}$ be homogeneous of degree 0. Then
\begin{enumerate}[label=$(\roman*)$]
	\item We have $\partial(f)=0$ if and only if $f$ is a chain map, and
	\item $f\in \operatorname{im}\partial$ if and only if $f$ is null-homotopic.
\end{enumerate}
\end{lemma}
By Lemma \ref{lemma:merkulov's construction}, the zeroth homology $H^0(\mathcal{D})$ describes precisely chain maps modulo homotopy, which implies that
$$H^0(\mathcal{D})= \operatorname{End}_{\mathcal{K}(\Lambda)}(P^\bullet)\cong \operatorname{Ext}_\Lambda^0(M,M)=\operatorname{End}_\Lambda(M)$$
as graded vector spaces. Considering instead the $k$th homology, $H^k(\mathcal{D})$, we obtain $H^k(\mathcal{D})\cong \operatorname{Ext}^k_\Lambda(M,M)$, and, consequently, $H^\ast(\mathcal{D})\cong \operatorname{Ext}^\ast_\Lambda(M,M)$.
We may decompose $\mathcal{D}$ as $\mathcal{D}=Z\oplus L$, where $Z$ is the graded subspace spanned by cycles and $L$ is some complement. Then, $Z$ decomposes further as $Z=H\oplus \operatorname{im}\partial$ where $\operatorname{im}\partial$ are the boundary maps and $H$ is some complement. Now, applying the first isomorphism theorem, we have
\begin{align*}\operatorname{Ext}_\Lambda^\ast(M,M)\cong H^\ast(\mathcal{D}) =\faktor{Z}{\operatorname{im}\partial}\cong H\end{align*}

and, from this, it follows that we have an injection $i:\operatorname{Ext}_\Lambda^\ast(M,M) \to H\subset \mathcal{D}$, such that the image of $\operatorname{Ext}_\Lambda^\ast(M,M)$ is isomorphic to a subspace of $\mathcal{D}$, which in turn is isomorphic to the homology. All this amounts to a decomposition $\mathcal{D}\cong \operatorname{im}\partial\oplus H \oplus L,$
where $H\cong \operatorname{Ext}_\Lambda^\ast(M,M)$.

Let $p:\mathcal{D}\to \operatorname{Ext}_\Lambda^\ast(M,M)$ be the composition given by the natural projection $p^\prime: \mathcal{D}\to H$ followed by the map $i^{-1}:H\to \operatorname{Ext}_\Lambda^\ast(M,M)$. Isolating degree parts, we have the following.
\begin{align*}\xymatrix@=0.3cm{
	\operatorname{im}\partial_{k-1}\oplus H_k \oplus L_k \ar[r]^-{\partial_k} &\operatorname{im}\partial_{k}\oplus H_{k+1} \oplus L_{k+1} \ar[r]^-{\partial_{k+1}} & \operatorname{im}\partial_{k+1} \oplus H_{k+2} \oplus L_{k+2}
}\end{align*}
Since $\operatorname{im}\partial_{k-1}$ and $H_k$ consist of cycles and $L_k$ is chosen to be a complement, we see that the differential $\partial$ maps
$\operatorname{im}\partial_{k-1}\oplus H_k\oplus L_k$ onto $\operatorname{im}\partial_{k}$ with kernel $\operatorname{im}\partial_{k-1}\oplus H_k$, so that $L_k\cong \operatorname{im}\partial_{k}$ via $\partial.$ This allows the definition of a map $h\in \mathcal{D}$ as below.
\begin{align*}\restr{h}{L_k\oplus H_k}=0,\quad \textrm{and} \quad \restr{h}{\operatorname{im}\partial_{k-1}}=\partial_{k-1}^{-1}.\end{align*}
Next, construct inductively a sequence  of linear maps $\lambda_n:\mathcal{D}^{\otimes n} \to \mathcal{D}$, such that $\deg \lambda_n=2-n$, as follows.

\begin{align*}h\lambda_1\coloneqq -\operatorname{id}_{\mathcal{D}},\quad \lambda_2\coloneqq \textrm{composition},\quad \lambda_n\coloneqq\sum_{r+s=n} (-1)^{s+1}\lambda_2(h\lambda_r\otimes h\lambda_s),\quad s\geq 1.\end{align*}

\begin{theorem}\label{theorem:merkulov model}\cite{LuPalmieriWuZhang,Merkulov}
Put $m_n \coloneqq p\lambda_n i^{\otimes n}.$ Then $(\operatorname{Ext}_\Lambda^\ast(M,M), m_2, m_3,\dots )$ is an $A_\infty$-algebra. Moreover, this structure is unique up to $A_\infty$-isomorphism.
\end{theorem}

When $M=\mathbb{L}$, there is the following nice characterization of when this process yields a trivial $A_\infty$-structure.
\begin{proposition}\label{proposition:ext-algebra of simples has trivial A-inf structure if koszul}\cite[Proposition~1]{Keller1}
	The $A_\infty$-algebra $\operatorname{Ext}_B^\ast(\mathbb{L},\mathbb{L})$ is $A_\infty$-isomorphic to an algebra with $m_n=0$ for $n\geq 3$ if and only if $B$ is Koszul.
\end{proposition}
Moreover, we crucially observe, that when the above construction is performed in the category of graded modules, the differential $\partial$ is of internal degree 0. This implies that also the map $h$ is of internal degree 0, which in turn implies that the maps $\lambda_n$ are of internal degree 0. This allows us to establish a sort of analogue of Proposition \ref{proposition:ext-algebra of simples has trivial A-inf structure if koszul} in the case where $M=\Delta$.

\begin{proposition}\label{prop:A_inf structure on dual ext alg of koszul algs is trivial}
	Let $B$ and $A$ be directed algebras with $B$ Koszul. Put $\Lambda=\mathcal{A}(B,A^{\operatorname{op}})$ and endow $\Lambda$ with the grading defined prior to Lemma \ref{lemma:homological and internal gradings on ext(delta,delta) coincide}. Let $\{m_n\}$ be the $A_\infty$-multiplications on $\operatorname{Ext}^\ast_\Lambda(\Delta,\Delta)$ constructed above. Then, we have $m_n=0$ for $n\neq 2$.
\end{proposition}
\begin{proof}
	By Theorem \ref{theorem:Lambda koszul iff A and B koszul}, $\Lambda$ is left standard Koszul. Since the maps $m_n$ constitute an $A_\infty$-structure, each $m_n$ is of homological degree $2-n$. Moreover, we saw that each $m_n$ is of internal degree 0. By Lemma \ref{lemma:homological and internal gradings on ext(delta,delta) coincide}, the two gradings coincide, so that $2-n=0.$
\end{proof}
\subsection{The general case}
Having dealt with the Koszul case, we return to the general case again. The construction used to obtain an $A_\infty$-structure on $\operatorname{Ext}_\Lambda^\ast(M,M)$ is not canonical, in the sense that there are several choices involved. Since our aim is to describe the $A_\infty$-structure on $\operatorname{Ext}_B^\ast(\Delta,\Delta)$ in terms of the $A_\infty$-structure on $\operatorname{Ext}_B^\ast(\mathbb{L},\mathbb{L})$, we start by investigating how the choices made when obtaining the $A_\infty$-structure on $\operatorname{Ext}_\Lambda^\ast(\Delta,\Delta)$ may be made compatible with those made to obtain the $A_\infty$-structure on $\operatorname{Ext}_B^\ast(\mathbb{L},\mathbb{L})$. Consider the following setup. Recall that, here, $F$ denotes the functor $\Lambda \otimes_B \blank : B\operatorname{-mod}\to \Lambda\operatorname{-mod}$. Let $P_B^\bullet\to \mathbb{L}$ be a minimal projective resolution, so that $F(P_B^\bullet)\to \Delta$ is a minimal projective resolution, too, by Lemma \ref{lemma:regular tensor simple iso standard & regular tensor projective iso projective,tensor preserves projective cover,lemma:tensor functors invertible on objects}. Let $d^P$ denote the differential on $P_B^\bullet$, so that $\operatorname{id}_\Lambda \otimes d^P$ is the differential on $F(P_B^\bullet)$. Then, we form the dg algebras $\mathcal{D}^B=\operatorname{End}_B(P_B^\bullet)$, and $\mathcal{D}^\Lambda=\operatorname{End}_\Lambda(F(P_B^\bullet))$. Let $\partial^B$ denote the differential on $\mathcal{D}^B$ and let $\partial^\Lambda$ denote the differential on $\mathcal{D}^\Lambda$.

Note that we have an injective homomorphism of algebras $\mathcal{D}^B\to \mathcal{D}^\Lambda$ given by $f\mapsto \operatorname{id}_\Lambda \otimes f$.
Here, we slightly abuse notation and identify $\Lambda\otimes_B L_B(i)\cong \Delta_B(i)$ and $\Lambda\otimes_B P_B(i)\cong P_\Lambda(i)$. We now perform Merkulov's construction to obtain decompositions of graded vector spaces
\begin{align*}\mathcal{D}^B= H^B\oplus \operatorname{im}\partial^B \oplus L^B,\quad \textrm{and}\quad \mathcal{D}^\Lambda=H^\Lambda \oplus \operatorname{im}\partial^\Lambda \oplus L^\Lambda,\end{align*}
where $H^B\cong \operatorname{Ext}_B^\ast(\mathbb{L},\mathbb{L})$, and $H^\Lambda\cong \operatorname{Ext}^\ast_\Lambda(\Delta,\Delta).$
Additionally, we obtain maps $i^B, i^\Lambda, p^B, p^\Lambda, h^B, h^\Lambda, \lambda_n^B$ and $\lambda_n^\Lambda$, so that the $A_\infty$-structures on $\operatorname{Ext}_B^\ast(\mathbb{L},\mathbb{L})$ and $\operatorname{Ext}_\Lambda^\ast(\Delta,\Delta)$ are given by
\begin{align*}m_n^B=p^B \lambda_n^B {(i^B)}^{\otimes n},\quad \textrm{and} \quad m_n^\Lambda =p^\Lambda \lambda_n^\Lambda {(i^\Lambda)}^{\otimes n},\end{align*}
respectively. We have the following useful observations.
\begin{lemma} \label{lemma:A-inf structure construction commutes with tensor functor}
	For any $f\in \mathcal{D}^B$, there holds
	\begin{enumerate}[label=$\roman*)$]
		\item $F ( \partial^B (f))=\partial^\Lambda( F (f))$.
		\item $\partial^B(f)=0$ if and only if $\partial^\Lambda( F (f))=0$.
		\item $f\in \operatorname{im}\partial^B$ if and only if $F(f)\in \operatorname{im}\partial^\Lambda$.
	\end{enumerate}
\end{lemma}
\begin{proof} Note that $\partial^\Lambda=F(\partial^B)=\operatorname{id}_\Lambda \otimes \partial^B$.
	\begin{enumerate}[label=$\roman*)$]
		\item We have \begin{align*}
		F\circ \partial^B(f)&=\operatorname{id}_\Lambda\otimes (d^Pf - (-1)^{|f|}fd^P)=(\operatorname{id}_\Lambda\otimes d^P)\left(\operatorname{id}_\Lambda \otimes f\right) - (-1)^{|\operatorname{id}_\Lambda\otimes f|}(\operatorname{id}_\Lambda\otimes f)(\operatorname{id}_\Lambda\otimes d^P)\\
		&=\partial^\Lambda ( F(f)),
		\end{align*}
		since $|f|=|\operatorname{id}_\Lambda\otimes f|$.
		\item Suppose $\partial^B(f)=0.$ Then, $\partial^\Lambda ( F(f))=F(\partial^B(f))=F(0)=0$,
		by $i)$.
		Suppose instead that $\partial^\Lambda( F(f))=0$. By $i)$, we have $0=\partial^\Lambda ( F(f))= F( \partial^B(f))$. Applying the functor $G$, we have $0=G \circ F (\partial^B(f))$. By Lemma \ref{lemma:regular tensor simple iso standard & regular tensor projective iso projective,tensor preserves projective cover,lemma:tensor functors invertible on objects}, $G\circ F$ is naturally isomorphic to $\operatorname{Id}_{B\operatorname{-mod}}$, implying there is a commutative diagram
	\begin{align*}\xymatrix@C=1.5cm{
			G\circ F (P^\bullet) \ar[r]^-{G\circ F (\partial^B (f))} \ar[d]_-{\psi} & G\circ F (P^\bullet) \ar[d]^-{\psi} \\
			P^\bullet \ar[r]_-{\partial^B(f)} & P^\bullet 
		}\end{align*}
		and since $\psi$ is an isomorphism, this implies $\partial^B(f)=0$.
		\item Suppose $f\in \operatorname{im}\partial^B$, so that $\partial^B(g)=f$ for some $g$. We claim that $\partial^\Lambda(F(g))=F(f)$. Indeed,
		\begin{align*}
		\partial^\Lambda\left( F(g) \right)&=(\operatorname{id}_\Lambda\otimes d^P)(\operatorname{id}_\Lambda \otimes g) - (-1)^{|\operatorname{id}_\Lambda\otimes g|}(\operatorname{id}_\Lambda \otimes g)(\operatorname{id}\otimes d^P)=\operatorname{id}_\Lambda\otimes d^Pg -(-1)^{|g|}(\operatorname{id}_\Lambda\otimes gd^P)\\
		&=\operatorname{id}_\Lambda\otimes (d^P g- (-1)^{|g|}gd^P)=F( \partial^B(g))=F(f).
		\end{align*} Suppose, instead, that $\partial^\Lambda(g)=F(f)$. For any $\alpha \in \mathcal{D}^B$, we have a commutative diagram
		\begin{align*}\xymatrixcolsep{1.5cm}\xymatrix{
			G\circ F (P^\bullet) \ar[r]^-{G\circ F (\alpha)} \ar[d]_-{\psi} & G\circ F (P^\bullet) \ar[d]^-{\psi} \\
			P^\bullet \ar[r]_-{\alpha} & P^\bullet 
		}\end{align*}
		according to Lemma \ref{lemma:regular tensor simple iso standard & regular tensor projective iso projective,tensor preserves projective cover,lemma:tensor functors invertible on objects}, so that $G\circ F(\alpha)=\psi^{-1}\alpha \psi$.
		We claim that $\partial^B( \psi G(g) \psi^{-1})=f$. Applying $G$ to $F(f)$, we get
		\begin{align*}
		GF(f)&=G(\partial^\Lambda (g))=G\left( (\operatorname{id}_\Lambda\otimes d^P) g - (-1)^{|g|} g (\operatorname{id}_\Lambda\otimes d^P)\right)=\\
		&=G(\operatorname{id}_\Lambda \otimes d^P) G(g) - (-1)^{|g|} G(g) G(\operatorname{id}_\Lambda \otimes d^P)=G F(d^P)G(g) - (-1)^{|g|} G(g) GF(d^P)\\
		&=(\psi^{-1} d^P \psi) G(g) - (-1)^{|g|} G(g) ( \psi^{-1} d^P \psi).
		\end{align*}
		Now, we use that $GF(f)=\psi^{-1}f\psi$ and compose with $\psi$ on the left and with $\psi^{-1}$ on the right, to get
		\begin{align*}
		f=\psi GF(f)\psi^{-1}=d^P \left(\psi G(g) \psi^{-1}\right) - (-1)^{|g|} \left(\psi G(g) \psi^{-1}\right) d^P&=\partial^B\left(\psi G(g) \psi^{-1}\right). \hfill \qedhere
		\end{align*}
	\end{enumerate}
\end{proof}
In what follows, let $\hat{\operatorname{rad}}(\Delta,\Delta)$ denote the space of chain map representatives of homomorphisms, contained in $\operatorname{rad}(\Delta,\Delta)$, which are of the form discussed prior to Proposition \ref{proposition:m_2 in ext-algebra}.

\begin{lemma}\label{lemma:form of hom composed with chain map}
 Let $\varepsilon\in \mathcal{D}^\Lambda$ be homogeneous map of degree $n$. Then, for any $f^\prime \in \hat{\operatorname{rad}}(\Delta,\Delta)$, the composition $f^\prime \circ \varepsilon$ has, at most, one non-zero component, and is given by a matrix, whose entries are maps of the form $\rho_{q^\prime}$, where $q^\prime$ is some linear combination of paths in $A^{\operatorname{op}}$.
\end{lemma}
\begin{proof}
We draw the composition $f^\prime \circ \varepsilon$. Recall that the chain map $f^\prime$ may be chosen to be of the below form, according to the discussion prior to Proposition \ref{proposition:m_2 in ext-algebra}.
$$\xymatrix{
\dots \ar[r] &P_{n+1} \ar[d]^-{\varepsilon_{n-1}} \ar[r] & P_n \ar[d]^-{\varepsilon_n} \ar[r] &\dots \\
\dots \ar[r]& P_1 \ar[r] \ar[d]^-0 & P_0 \ar[d]^-{\tilde{f}} \ar[r] & 0 \\
\dots \ar[r]& P_1 \ar[r] & P_0 \ar[r] & 0
}$$
From this picture it is clear that $f^\prime\circ \varepsilon$ has at most one non-zero component, namely, the homomorphism $\tilde{f}\varepsilon_n:P_n\to P_0$. Fix a decomposition of $P_n$ and $P_0$ into indecomposable projective modules. We already know that the matrix of $\tilde{f}$, with respect to the above decomposition, has entries of the form $\rho_{q^\prime}$, where $q^\prime$ is a linear combination of paths in $A^{\operatorname{op}}$. The entries of the matrix of $\varepsilon_n$, with respect to the above decomposition, are homomorphisms between indecomposable projective $\Lambda$-modules. There are three cases.

\begin{enumerate}[label=$(\roman*)$]
	\item The matrix of $\varepsilon_n$ has an entry $\hat{\varepsilon}_n:P_\Lambda(x)\to P_\Lambda(x)$, for some vertex $x$. Then, $\hat{\varepsilon}_n=a1_{P_\Lambda(x)} + b\rho_{v}$, where $v$ is some linear combination of paths in $\Lambda$ and $a,b\in K$ are some scalars. Note that the paths occuring in $v$ are of the form $w^\prime u$, with both $u$ and $w^\prime$ being non-trivial paths, in $A^{\operatorname{op}}$ and $B$, respectively. The resulting entry of the matrix of the composition, $f^\prime\circ \varepsilon$, is  $$\rho_{q^\prime}(a1_{P_\Lambda(x)}+b\rho_{v})=a\rho_{q^\prime} + b\rho_{v q^\prime}=a\rho_{q^\prime},$$ since $vq^\prime=0$, because $w^\prime u q^\prime=0$, according to the dual extension relation.
	\item The matrix of $\varepsilon_n$ has an entry $\hat{\varepsilon}_n: P_\Lambda(x)\to P_\Lambda(y)$, with $x>y$. Then, $\hat{\varepsilon}_n=\rho_{v}$, where $v$ is a linear combination of paths from $y$ to $x$. Again, the paths occuring in $v$ are of the form $w^\prime u$, where $w^\prime$ and $u$ are as in the previous case. The resulting entry of the matrix of the composition, $f^\prime\circ\varepsilon$, is
	$$\rho_{q^\prime}\rho_{v}=\rho_{v q^\prime}=0,$$
	since $vq^\prime=0$, because $w^\prime u q^\prime=0$, according to the dual extension relation.
	\item The matrix of $\varepsilon_n$ has an entry $\hat{\varepsilon}_n: P_\Lambda(x)\to P_\Lambda(y)$, wth $x<y$. Then, $\hat{\varepsilon}_n=\rho_{u^\prime}$, where $u^\prime$ is some linear combination of paths, from $y$ to $x$, in $A^{\operatorname{op}}$. Note that the paths constituting $u^\prime$ go in decreasing direction, and are therefore paths in $A^{\operatorname{op}}$. The resulting entry of the matrix of the composition, $f^\prime\circ \varepsilon$, is $\rho_{q^\prime}\rho_{v^\prime}=\rho_{v^\prime q^\prime}$.\hfill \qedhere
\end{enumerate}
\end{proof}
\begin{lemma}\label{lemma:radical hat composed with induced part isomorphic to}
	There is an isomorphism of vector spaces $\hat{\operatorname{rad}}(\Delta,\Delta)\circ F(H^B)\cong \operatorname{rad}(\Delta,\Delta)\circ F(\operatorname{Ext}_B^\ast(\mathbb{L},\mathbb{L}))$.
\end{lemma}
\begin{proof}
	Let $\chi^B:H^B\to \operatorname{Ext}_B^\ast(\mathbb{L},\mathbb{L})$ be the inverse of the isomorphism $i^B:\operatorname{Ext}_B^\ast(\mathbb{L},\mathbb{L})\to H^B$. Since
	$$F:H^B \to F(H^B)\quad\textrm{and}\quad F:\operatorname{Ext}_B^\ast(\mathbb{L},\mathbb{L})\to F(\operatorname{Ext}_B^\ast(\mathbb{L},\mathbb{L}))$$
	are linear isomorphisms, we have an isomorphism of vector spaces $F\chi^B F^{-1}: F(H^B)\to F(\operatorname{Ext}_B^\ast(\mathbb{L},\mathbb{L}))$.
	
	For a homomorphism $f^\prime\in \operatorname{rad}(\Delta,\Delta)$, let $\hat{f}$ denote its chain map representative in $\hat{\operatorname{rad}}(\Delta,\Delta)$. We claim that any two  maps in $\hat{\operatorname{rad}}(\Delta,\Delta)\circ F(H^B)$ are not homotopic. Indeed, assume that $\hat{f}\varepsilon$ and $\hat{g}\delta$ are homotopic. Consider the following picture.
	$$\xymatrix{
		P_{n+1} \ar[d]_-0\ar[r]^-{d_{n}} & P_n \ar[ld]_-{h_n}\ar[r]^-{d_{n-1}} \ar[d]^{\hat{f}\varepsilon_n-\hat{g}\delta_n} & P_{n-1} \ar@/^1pc/[ld]^<<<<<{h_{n-1}}\\
		P_1\ar[r]_{d_1}& P_0\ar[r] & 0
	}$$
	The entries of the matrix of the map $d_1h_n-h_{n-1}d_{n-1}$ are of the form $\rho_{q^\prime p}$, where $q^\prime p$ is some linear combination of paths in $\Lambda$. Since $\partial^\Lambda=\operatorname{id}_\Lambda \otimes \partial^B$, all paths in $q^\prime p$ contain non-trivial subpaths in $B$. At the same time, by Lemma \ref{lemma:form of hom composed with chain map}, the entries of the matrix of the map $\hat{f}\varepsilon_n-\hat{g}\delta_n$ are of the form $\rho_{v^\prime}$, where $v^\prime$ is a linear combination of paths in $A^{\operatorname{op}}$. Since any path containing a non-trivial subpath in $B$ is linearly independent from any path in $A^{\operatorname{op}}$, we cannot have $d_1h_n+h_{n-1}d_{n-1}=f\varepsilon_n-g\delta_n$ unless $\hat{f}\varepsilon_n-\hat{g}\delta_n=0$, which is equivalent to $\hat{f}\varepsilon=\hat{g}\delta$.
	
	Now, define a linear map $\varphi: \hat{\operatorname{rad}}(\Delta,\Delta)\circ F(H^B)\to \operatorname{rad}(\Delta,\Delta)\circ F(\operatorname{Ext}_B^\ast(\mathbb{L},\mathbb{L})$ by
	$$\hat{f}\circ F(\varepsilon)\mapsto f^\prime\circ F\chi^BF^{-1} F(\varepsilon)=f^\prime\circ F\chi^B(\varepsilon).$$
 It is immediately clear that $\varphi$ is surjective. Moreover, $\varphi$ is injective. To see this, we observe, that if we have $f^\prime\circ F\chi^B(\varepsilon)=g^\prime\circ F\chi^B(\delta)$, then, the chain maps $\hat{f}\circ F(\varepsilon)$ and $\hat{g}\circ F(\delta)$ both lift the extension $f^\prime\circ F\chi^B(\varepsilon)$ to the homotopy category. Consequently, they are homotopic. By the above argument, we must then have $\hat{f}\circ F(\varepsilon)=\hat{g}\circ F(\delta)$. 
\end{proof}
With these elementary properties established, we are ready to prove the following key proposition. The inspiration for this result is Theorem~3.21 in \cite{2021uniquenessKM}, which is similar, albeit dealing with a more general context.
\begin{proposition}\label{proposition:A-inf structure constructions compatible}
	Suppose we perform Merkulov's construction on $\mathcal{D}^B$, to obtain the decomposition $\mathcal{D}^B=H^B\oplus \operatorname{im} \partial^B \oplus L^B$. Then, we may perform Merkulov's construction on $\mathcal{D}^\Lambda$, to obtain a decomposition $\mathcal{D}^\Lambda=H^\Lambda \oplus \operatorname{im}\partial^\Lambda \oplus L^\Lambda$,
	such that
	$$H^\Lambda=F(H^B)\oplus (\hat{\operatorname{rad}}(\Delta,\Delta)\circ F(H^B)),\quad \operatorname{im}\partial^\Lambda=F(\operatorname{im}\partial^B) \oplus \partial^\Lambda(\hat{L})\quad \textrm{and}\quad L^\Lambda=F(L^B)\oplus \hat{L},$$
	where $\hat{\operatorname{rad}}(\Delta,\Delta)$ is the space consisting of chain map representatives of homomorphisms in $\operatorname{rad}(\Delta,\Delta)$, which are of the form described in the discussion prior to Proposition \ref{proposition:m_2 in ext-algebra}.
\end{proposition}
\begin{proof} Suppose we have $\mathcal{D}^B=Z^B\oplus L^B$, where $Z^B$ is the subspace spanned by cycles and $L^B$ is some complement. Let $Z^\Lambda\subset \mathcal{D}^\Lambda$ be the subspace spanned by cycles. Then, there holds $F(Z^B)\subset Z^\Lambda$, by Lemma \ref{lemma:A-inf structure construction commutes with tensor functor}. We claim that $F(L^B)\cap Z^\Lambda=\{0\}$. To see this, let $f\in F(L^B)\cap Z^\Lambda$. Since $f$ is in the image of $F$, we have $f=\operatorname{id}_\Lambda\otimes g$, for some $g\in \mathcal{D}^B$. Then $\partial^\Lambda(\operatorname{id}_\Lambda\otimes g)=0$, which implies $\partial^B(g)=0$, so that $g\in Z^B\cap L^B$, which, in turn, implies $g=0$. Then, we also have $f=0$. This implies that we have a decomposition 
	\begin{align*}\mathcal{D}^\Lambda=Z^\Lambda \oplus ( F(L^B) \oplus \hat{L}),\end{align*}
	for some $\hat{L}$. Next, we claim that $\operatorname{im}\partial^\Lambda=F(\operatorname{im}\partial^B ) \oplus \partial^\Lambda (\hat{L}).$ To see this, we note that $F(\operatorname{im}\partial^B)\cap \partial^\Lambda(\hat{L})=\{0\}.$
	Indeed, suppose $f\in F(\operatorname{im}\partial^B)\cap \partial^\Lambda(\hat{L})$. Then, we have $f=\operatorname{id}_\Lambda \otimes \partial^B(g)=\partial^\Lambda(\operatorname{id}_\Lambda\otimes g)$. If $f\neq 0$, we must have $\partial^B(g)\neq 0$. We may then write $g=g_1+g_2$ with $g_1\in L^B$ and $g_2\in Z^B$. At the same time, we have $f=\partial^\Lambda(\hat{g})$ for some $\hat{g} \in \hat{L}$. Now we have $f=\partial^\Lambda(\hat{g})=\partial^\Lambda(\operatorname{id}_\Lambda\otimes g_1)$
	with $\hat{g}, \operatorname{id}_\Lambda \otimes g_1 \in L^\Lambda$. Since $\partial^\Lambda$ restricts to an isomorphism on $L^\Lambda$, we get $\hat{g}=\operatorname{id}_\Lambda \otimes g_1=0$ since $F(L^B)\cap \hat{L}=\{0\}$, and, consequently, $f=0$. Then, we see that
	
	\begin{align*}\operatorname{im}\partial^\Lambda&=\partial^\Lambda( F(L^B) \oplus \hat{L})=\partial^\Lambda F(L^B) + \partial^\Lambda(\hat{L})=F(\partial^B(L^B))+\partial^\Lambda(\hat{L})=F(\operatorname{im}\partial^B) \oplus \partial^\Lambda(\hat{L}),
	\end{align*}
where, in the last equality, we use that the intersection $F(\operatorname{im}\partial^B ) \cap \partial^\Lambda(\hat{L})$ is zero, to conclude that the sum on the right is direct.
	This gives us a decomposition $\mathcal{D}^\Lambda=H^\Lambda \oplus ( F(\operatorname{im}\partial^B) \oplus \partial^\Lambda(\hat{L}) ) \oplus (F(L^B)\oplus \hat{L})$. Next, we claim that $F(H^B)\cap (F(\operatorname{im}\partial^B) \oplus \partial^\Lambda(\hat{L}))=\{0\}$. Let $\operatorname{id}_\Lambda\otimes f\in F(H^B)\cap(F(\operatorname{im}\partial^B) \oplus \partial^\Lambda(\hat{L}))$. Then, we must have $\operatorname{id}_\Lambda\otimes f=\operatorname{id}_\Lambda \otimes \partial^B(g) + \partial^\Lambda(\hat{g})$ for $f\in H^B, g\in L^B$ and $\hat{g}\in \hat{L}$. But, then, we get 
	\begin{align*} \operatorname{id}_\Lambda\otimes f&=\operatorname{id}_\Lambda \otimes \partial^B(g) + \partial^\Lambda(\hat{g})=\partial^\Lambda( \operatorname{id}_\Lambda\otimes g + \hat{g}),
	\end{align*}
	which implies that $f\in \operatorname{im}\partial^B$, by Lemma \ref{lemma:A-inf structure construction commutes with tensor functor}. This means that we must have $f=0$, since $f\in H^B\cap \operatorname{im}\partial^B$. This fact, in turn, implies that we may choose a complement $\hat{H}$ of $F(H^B)\oplus \operatorname{im}\partial^\Lambda$ in $Z^\Lambda$, giving us a decomposition
	\begin{align*}\mathcal{D}^\Lambda&=( F(H^B) \oplus \hat{H})\oplus (F(\operatorname{im}\partial^B) \oplus \partial^\Lambda(\hat{L})) \oplus (F(L^B)\oplus \hat{L}).\end{align*}
	
	Next, we want to show that it is possible to choose $\hat{H}=\hat{\operatorname{rad}}(\Delta,\Delta)\circ F(H^B)$. To this end, we claim, that for $f\in \hat{\operatorname{rad}}(\Delta,\Delta)\circ F(H^B)$, we have  $\partial^\Lambda(f)=0$. Consider the following picture.
	$$\xymatrix{
		P_{n+1} \ar[d]^-0\ar[r]^-{d_{n+1}} & P_n \ar[r]^-{d_{n}}  \ar[d]^-{\tilde{f}} & P_{n-1}\\
		P_1\ar[r]& P_0\ar[r] & 0
	}$$
Due to $f$ having only one non-zero component, the only possibly non-zero component of the map $\partial^\Lambda(f)$ is $\tilde{f} d_{n+1}$. Since the entries of the matrix of $\tilde{f}$ are right multiplications by paths in $A^{\operatorname{op}}$, and the entries of the matrix of $d_{n+1}$ are right multiplications by linear combinations of paths in $B$, the claim $\tilde{f}d_{n+1}=0$ now follows from the fact that $\rho_{q^\prime}\rho_p=\rho_{pq^\prime}=0$, where we use that $pq^\prime=0$, according to the dual extension relation.

Next, we check that $(\hat{\operatorname{rad}}(\Delta,\Delta) \circ F(H^B))\cap (F(H^B)\oplus F(\operatorname{im}\partial^B) \oplus \partial^\Lambda(\hat{L}))=\{0\}$. Let $f=f_1+f_2+f_3$ be contained in this intersection, where $f\in \hat{\operatorname{rad}}(\Delta,\Delta)\circ F(H^B), f_1\in F(H^B), f_2\in F(\operatorname{im}\partial^B)$, and $f_3\in \partial^\Lambda(\hat{L})$. Since $f\in \hat{\operatorname{rad}}(\Delta,\Delta)\circ F(H^B)$, we apply Lemma \ref{lemma:form of hom composed with chain map} to see that $f$ has at most one non-zero component. Therefore, we compare the maps $\tilde{f}=\tilde{f}_1+\tilde{f}_2+\tilde{f}_3:P_n\to P_0$. Decomposing $P_n$ and $P_0$ into indecomposable summands, the map $\tilde{f}=\tilde{f}_1+\tilde{f}_2+\tilde{f}_3$ is given by a matrix. Considering a non-zero entry, $\hat{f}$, of this matrix, we write $\hat{f}=\hat{f}_1+\hat{f}_2+\hat{f}_3$. Since $f_1\in F(H^B)$, we have $\hat{f}_1=\rho_{p_1}$ for some linear combination of paths, $p_1$, in $B$. Similarly, we have $\hat{f}_2=\rho_{p_2}$, for some linear combination of paths, $p_2$, in $B$. Since $\partial^\Lambda=\operatorname{id}_\Lambda \otimes \partial^B$, and $f_3\in \partial^\Lambda(\hat{L})$, we have $\hat{f}_3=\rho_{w}$, where $w$ is a path in $\Lambda$ containing a non-trivial subpath in $B$. But any path in $A^{\operatorname{op}}$ is linearly independent from $p_1, p_2$ and $w$, which implies $f=0$.

This shows that a complement $\hat{H}$ of $F(H^B)$, in the homology $H^\Lambda$, may be chosen so that the subspace $\hat{\operatorname{rad}}(\Delta,\Delta)\circ F(H^B)$ is contained in $\hat{H}$. It is now enough to show that the space $\hat{\operatorname{rad}}(\Delta,\Delta)\circ F(H^B)$ has the ``correct'' dimension.

By construction, the homology $H^\Lambda$ is isomorphic to $\operatorname{Ext}_\Lambda^\ast(\Delta,\Delta)$. In the proof of Theorem \ref{theorem:ext-algebra of standards isomorphic to dual extension of ext-algebra of simples and B}, we saw that $\operatorname{Ext}_\Lambda^\ast(\Delta,\Delta)$ is generated by the extensions $F(\operatorname{Ext}_B^\ast(\mathbb{L}, \mathbb{L}))$ together with the homomorphisms $\operatorname{rad}(\Delta,\Delta)$. If $\mathcal{B}$ and $\mathcal{B}^\prime$ are bases of $\operatorname{Ext}_B^\ast(\mathbb{L},\mathbb{L})$ and of $\operatorname{rad}(\Delta,\Delta)$, respectively, then, the non-zero elements of the set $\{\varepsilon, f\circ \varepsilon\mid\varepsilon\in \mathcal{B},f\in \mathcal{B}^\prime\}$ form a basis of $\operatorname{Ext}_\Lambda^\ast(\Delta,\Delta)$, as shown in the proof of Theorem \ref{theorem:ext-algebra of standards isomorphic to dual extension of ext-algebra of simples and B}. This implies that
$$\dim \operatorname{Ext}_\Lambda^\ast(\Delta,\Delta)=\dim F(\operatorname{Ext}_B^\ast(\mathbb{L},\mathbb{L}) + \dim (\operatorname{rad}(\Delta,\Delta)\circ F(\operatorname{Ext}_B^\ast(\mathbb{L},\mathbb{L}))=\dim F(H^B) + \dim \hat{H},$$
which, in turn, implies that $\dim \hat{H}=\dim (\operatorname{rad}(\Delta,\Delta)\circ F(\operatorname{Ext}_B^\ast(\mathbb{L},\mathbb{L})).$ Now we appeal to Lemma \ref{lemma:radical hat composed with induced part isomorphic to}, the statement of which is equivalent to
$$\dim (\operatorname{rad}(\Delta,\Delta)\circ F(\operatorname{Ext}_B^\ast(\mathbb{L},\mathbb{L}))=\dim ( \hat{\operatorname{rad}}(\Delta,\Delta)\circ F(H^B)),$$
completing the proof. \qedhere
\end{proof}
\begin{corollary}\label{corollary:proposition:A-inf structure constructions compatible 1}
	For any $f\in \mathcal{D}^B$, there holds $F\circ p^B (f)=p^\Lambda \circ F(f)$.	
\end{corollary}
\begin{proof} Let $\chi^B:H^B\to \operatorname{Ext}_B^\ast(\mathbb{L},\mathbb{L})$ be the inverse of the isomorphism $i^B: \operatorname{Ext}_B^\ast(\mathbb{L},\mathbb{L})\to H^B$. Similarly, let $\chi^\Lambda:H^\Lambda\to \operatorname{Ext}_B^\ast(\mathbb{L},\mathbb{L})$ be the inverse of the isomorphism $i^\Lambda: \operatorname{Ext}_\Lambda^\ast(\Delta,\Delta)\to H^\Lambda$. Recall that, by definition, we have $p^B=\chi^B \pi^B$ and $p^\Lambda=\chi^\Lambda \pi^\Lambda$, where $\pi^B$ and $\pi^\Lambda$ are as below. Consider the following diagram.
	\begin{align*}\xymatrix{
			H^B \oplus \operatorname{im} \partial^B \oplus L^B \ar[r]^-{\pi^B} \ar[d]_-{F} & H^B \ar[d]_-{F} \ar[r]^-{\chi^B} & \operatorname{Ext}_B^\ast(\mathbb{L},\mathbb{L}) \ar[d]^-{F}\\
			H^\Lambda \oplus \operatorname{im} \partial^\Lambda \oplus L^\Lambda \ar[r]^-{\pi^\Lambda} & H^\Lambda \ar[r]^-{\chi^\Lambda} & \operatorname{Ext}_\Lambda^\ast(\Delta,\Delta)
	}\end{align*}
	The left square commutes, by Proposition \ref{proposition:A-inf structure constructions compatible}. We need to show that there is a choice of the isomorphism $\chi^\Lambda$ such that the right square commutes. According to Proposition \ref{proposition:A-inf structure constructions compatible}, we may write 
	$$H^\Lambda=F(H^B) \oplus (\hat{\operatorname{rad}}(\Delta,\Delta)\circ F(H^B)),\quad\textrm{and}\quad \operatorname{Ext}_\Lambda^\ast(\Delta,\Delta)=F(\operatorname{Ext}_B^\ast(\mathbb{L},\mathbb{L}))\oplus (\operatorname{rad}(\Delta,\Delta)\circ F(\operatorname{Ext}_B^\ast(\mathbb{L},\mathbb{L}))).$$
	
	Since $F:H^B\to F(H^B)$ and $F:\operatorname{Ext}_B^\ast(\mathbb{L},\mathbb{L}) \to F(\operatorname{Ext}_B^\ast(\mathbb{L},\mathbb{L}))$ are isomorphisms, we may define an isomorphism $\psi :F(H^B)\to F(\operatorname{Ext}_B^\ast(\mathbb{L},\mathbb{L}))$ by $\psi=F\chi^BF^{-1}$. Let $\varphi:\hat{\operatorname{rad}}(\Delta,\Delta)\circ F(H^B)\to \operatorname{rad}(\Delta,\Delta)\circ F(\operatorname{Ext}_B^\ast(\mathbb{L},\mathbb{L}))$ be the isomorphism constructed in the proof of Lemma \ref{lemma:radical hat composed with induced part isomorphic to}. Define the isomorphism $$\chi^\Lambda: F(H^B)\oplus (\hat{\operatorname{rad}}(\Delta,\Delta)\circ F(H^B))\to F(\operatorname{Ext}_B^\ast(\mathbb{L},\mathbb{L})\oplus (\operatorname{rad}(\Delta,\Delta)\circ F(\operatorname{Ext}_B^\ast(\mathbb{L},\mathbb{L}))$$
	as the matrix $\chi^\Lambda=\left[\begin{smallmatrix}
		\psi & 0 \\ 0 & \varphi
	\end{smallmatrix}\right]$, and consider the following diagram.
	
	\begin{align*}\xymatrix{
			H^B \ar[d]_-{\left[\begin{smallmatrix}F \\ 0\end{smallmatrix}\right]} \ar[r]^-{\chi^B} & \operatorname{Ext}_B^\ast(\mathbb{L},\mathbb{L}) \ar[d]^-{\left[\begin{smallmatrix}F \\ 0\end{smallmatrix}\right]} \\	
			F(H^B) \oplus (\hat{\operatorname{rad}}(\Delta,\Delta)\circ F(H^B)) \ar[r]_-{\left[\begin{smallmatrix}
					\psi & 0 \\ 0 & \varphi
				\end{smallmatrix}\right]} & F(\operatorname{Ext}_B^\ast (\mathbb{L},\mathbb{L})) \oplus (\operatorname{rad}(\Delta,\Delta)\circ \operatorname{Ext}_B^\ast(\mathbb{L},\mathbb{L}))
	}\end{align*}
	Now, for any $f\in H^B$, we have
	\begin{align*}
		\begin{bmatrix}
			\psi & 0 \\ 0 & \varphi
		\end{bmatrix}\begin{bmatrix}
			F\\0
		\end{bmatrix}f&=\begin{bmatrix}
			\psi & 0 \\ 0 & \varphi 
		\end{bmatrix}\begin{bmatrix}
			F(f) \\ 0
		\end{bmatrix}=\begin{bmatrix}
			\psi F (f) \\ 0
		\end{bmatrix}=\begin{bmatrix}
			F\chi^B (f)\\0
		\end{bmatrix}=\begin{bmatrix}
			F \\ 0
		\end{bmatrix}\chi^B(f). \qedhere
	\end{align*}
\end{proof}
\begin{corollary} \label{corollary:proposition:A-inf structure constructions compatible 2}
For any $f\in \mathcal{D}^B$, there holds $F\circ h^B(f)=h^\Lambda\circ F(f)$.
\end{corollary}
\begin{proof}
	We consider the diagram
	\begin{align*}\xymatrix{H^B\oplus \operatorname{im}\partial^B \oplus L^B \ar[r]^-F \ar[d]_-{h^B} & \left( F(H^B) \oplus  (\hat{\operatorname{rad}}(\Delta,\Delta)\circ F(H^B))\right)\oplus \left(F(\operatorname{im}\partial^B) \oplus \partial^\Lambda(\hat{L})\right) \oplus \left(F(L^B)\oplus \hat{L}\right) \ar[d]^-{h^\Lambda} \\
		L^B \ar[r]_-{F} & F(L^B)}\end{align*}
	Supose $f\in H^B \oplus L^B$. Then the bottom path is immediately zero, and the top path is too, because we have $F(f)\in H^\Lambda \oplus L^\Lambda$, which implies that $h^\Lambda F(f)=0$. If, instead, $f\in \operatorname{im}\partial^B$, put $f=\partial^B(g)$ for some $g\in L^B$. Then, the bottom path equals $F(g)$. For the top path, we get $h^\Lambda F(f)=h^\Lambda F(\partial^B(g))=h^\Lambda \partial^\Lambda F(g)=F(g)$ since $F(g)\in F(L^B)$.
\end{proof}

Armed with Proposition \ref{proposition:A-inf structure constructions compatible} and its corollaries, we are ready to show that the $A_\infty$-structures on $\operatorname{Ext}_B^\ast(\mathbb{L},\mathbb{L})$ and $\operatorname{Ext}_\Lambda^\ast(\Delta,\Delta)$ obtained from our process respect the embedding $\operatorname{Ext}_B^\ast(\mathbb{L},\mathbb{L}) \hookrightarrow \operatorname{Ext}_\Lambda^\ast(\Delta,\Delta)$.
\begin{proposition}\label{proposition:A-inf structure compatible with embedding}
	Let $n\geq 2$. For any $\varepsilon_1,\dots, \varepsilon_n\in \operatorname{Ext}_B^\ast(\mathbb{L},\mathbb{L})$, there holds the following formula.
	\begin{align*}F\left( m_n^B(\varepsilon_1,\dots, \varepsilon_n)\right)=m_n^\Lambda\left(F(\varepsilon_1),\dots, F(\varepsilon_n)\right).\end{align*}
\end{proposition}
\begin{proof}
	We begin by proving the corresponding statement for the maps $\lambda_n^B$ and $\lambda_n^\Lambda$, that is, 
		\begin{align}\label{equation: lambda_n respects embedding}F\left( \lambda_n^B(\varepsilon_1,\dots, \varepsilon_n)\right)=\lambda_n^\Lambda\left(F(\varepsilon_1),\dots, F(\varepsilon_n)\right).\end{align}
		Note that, in the above formula, we have abused notation and identified the extensions $\varepsilon_1,\dots,\varepsilon_n$ with their chain map representatives. We proceed by induction. For $n=2$, we have
		$$F(\lambda_2^B(\varepsilon_1,\varepsilon_2))=F(\varepsilon_1 \varepsilon_2)=F(\varepsilon_1)F(\varepsilon_2)=\lambda_2^\Lambda(F(\varepsilon_1)F(\varepsilon_2)),$$
		since the maps $\lambda_2^B$ and $\lambda_2^\Lambda$ are defined as composition. The basis of the induction is thus clear. Let $n\geq 3$ and suppose the formula \eqref{equation: lambda_n respects embedding} holds for all $r, s < n$. Then, adopting the convention that $h^\Lambda F \lambda_1^B=-F$, we have 
	\begin{align*}
	\lambda_n^\Lambda(F(\varepsilon_1),\dots, F(\varepsilon_n))&=\sum_{r+s=n} (-1)^{s+1}\lambda^\Lambda_2(h^\Lambda\lambda^\Lambda_r(F(\varepsilon_1),\dots, F(\varepsilon_r))\otimes h^\Lambda\lambda^\Lambda_s(F(\varepsilon_{r+1}),\dots, F(\varepsilon_n)))\\
	&=\sum_{r+s=n} (-1)^{s+1} \lambda_2^\Lambda(h^\Lambda (F( \lambda_r^B(\varepsilon_1,\dots, \varepsilon_r)))\otimes h^\Lambda(F( \lambda_s^B(\varepsilon_{r+1},\dots, \varepsilon_n))) \\
	&=\sum_{r+s=n}(-1)^{s+1}\lambda_2^\Lambda( F(h^B\lambda_r^B(\varepsilon_1,\dots, \varepsilon_r))\otimes F(h^B\lambda_s^B(\varepsilon_{r+1},\dots, \varepsilon_n)))\\
	&=\sum_{r+s=n} (-1)^{s+1}F(\lambda_2^B (h^B\lambda_r^B(\varepsilon_1,\dots, \varepsilon_r)\otimes h^B\lambda_s^B(\varepsilon_{r+1},\dots, \varepsilon_n)))\\
	&=F \left(\sum_{r+s=n} (-1)^{s+1}\lambda_2^B (h^B\lambda_r^B(\varepsilon_1,\dots, \varepsilon_r)\otimes h^B\lambda_s^B(\varepsilon_{r+1},\dots, \varepsilon_n))\right)\\
	&=F( \lambda_n^B(\varepsilon_1,\dots, \varepsilon_n)),
	\end{align*}
where we additionally use Corollary \ref{corollary:proposition:A-inf structure constructions compatible 2}. This fact, together with Corollary \ref{corollary:proposition:A-inf structure constructions compatible 1}, then implies
	\begin{align*}
	m_n^\Lambda(F(\varepsilon_1),\dots F(\varepsilon_n))&=p^\Lambda \lambda_n^\Lambda ({i^\Lambda})^{\otimes n} (F(\varepsilon_1),\dots, F(\varepsilon_n))=p^\Lambda \lambda_n^\Lambda(F(\varepsilon_1),\dots, F(\varepsilon_n))=p^\Lambda (F(\lambda_n^B(\varepsilon_1,\dots, \varepsilon_n)))\\
	&=F( p^B \lambda_n^B(\varepsilon_1,\dots, \varepsilon_n))=F( p^B \lambda_n^B{i^B}^{\otimes n}(\varepsilon_1,\dots, \varepsilon_n))=F( m_n^B(\varepsilon_1,\dots, \varepsilon_n)). \qedhere
	\end{align*}
\end{proof}
Proposition \ref{proposition:A-inf structure compatible with embedding} describes how the $A_\infty$-multiplications on $\operatorname{Ext}_\Lambda^\ast(\Delta,\Delta)$ behave when their arguments are extensions of the form $F(\varepsilon)$ for $\varepsilon\in \operatorname{Ext}_B^\ast(\mathbb{L},\mathbb{L})$. But, as we know from Theorem \ref{theorem:ext-algebra of standards isomorphic to dual extension of ext-algebra of simples and B}, the algebra $\operatorname{Ext}_\Lambda^\ast(\Delta,\Delta)$ contains extensions which are not of this form.

To be able to say something about the $A_\infty$-structure on the whole of $\operatorname{Ext}_\Lambda^\ast(\Delta,\Delta)$, we make the following technical assumption. Assume in what follows, that in the decomposition $\mathcal{D}^B=H^B\oplus \operatorname{im}\partial^B \oplus L^B$, the space $L^B$ may be chosen in such a way that all components of any map $\varepsilon\in L^B$ are radical maps, that is, the components of $\varepsilon\in L^B$ are given by matrices, whose entries are of the form $\rho_p$, where $p$ is a linear combination of non-trivial paths in $B$.

\begin{lemma}\label{lemma: L^Lambda consists of only radical maps}
In the decomposition $\mathcal{D}^\Lambda=H^\Lambda \oplus \operatorname{im}\partial^\Lambda \oplus L^\Lambda$, the space $L^\Lambda$ may be chosen in such a way that all components of any map in $L^\Lambda$ are given by matrices whose entries are of the form $\rho_{q^\prime p}$, where $p$ is a linear combination of non-trivial paths in $B$ and $q^\prime$ is some linear combination of paths in $A^{\operatorname{op}}$.
\end{lemma}
\begin{proof}
	Consider $\varepsilon\in \mathcal{D}^\Lambda$, with $\varepsilon$ homogeneous of degree $n$. Recall that, if $P^\bullet \to \mathbb{L}$ is a minimal projective resolution, then so is $F(P^\bullet)\to \Delta$. We write $\varepsilon=\alpha+\beta+\gamma$, where $\alpha \in H^\Lambda, \beta\in \operatorname{im}\partial^\Lambda$ and $\gamma\in L^\Lambda$. Consider the following picture.
	
	$$\varepsilon:\vcenter{\xymatrix{
	F(P_{n+1}) \ar[d]^-{\varepsilon_{n+1}}\ar[r] & F(P_n) \ar[d]^-{\varepsilon_n}\ar[r] & F(P_{n-1})	\\
	F(P_1)\ar[r] & F(P_0)\ar[r] & 0
	}}$$
Fix decompositions of $P_n$ and $P_0$ into direct sums of indecomposable projective $B$-modules. This induces natural decompositions of $F(P_n)$ and $F(P_0)$ into direct sums of indecomposable projective $\Lambda$-modules. Assume now that the matrix of the map $\varepsilon_n$ has a non-zero entry, $f$, at some fixed position, which is not a radical map. Then, we may write $f=\mu \cdot 1_{P_\Lambda(x)} + f^\prime$, where $\mu\in K$ is some scalar and $f^\prime\in \operatorname{rad}(P_\Lambda(x),P_\Lambda(x))$.

Consider now instead the map $\overline{\varepsilon}\in \mathcal{D}^B$, homogeneous of degree $n$, such that the entry at the fixed position of the matrix of $\overline{\varepsilon}_n$ is equal to $\mu \cdot 1_{P_B(x)}$ and all other entries are 0. Put $\overline{\varepsilon}_k=0$ for $n\neq k$.

$$\overline{\varepsilon}: \vcenter{\xymatrix{
	P_{n+1} \ar[d]^-{0}\ar[r] & P_n \ar[d]^-{\overline{\varepsilon}_n=\overline{\alpha}_n+\overline{\beta}_n+\overline{\gamma}_n}\ar[r] & P_{n-1}	\\
	P_1\ar[r] & P_0\ar[r] & 0
}}$$
Now, we write $\overline{\varepsilon}=\overline{\alpha}+\overline{\beta}+\overline{\gamma}$, where $\overline{\alpha}\in H^B, \overline{\beta}\in \operatorname{im}\partial^B$ and $\overline{\gamma}\in L^B$. The maps in $L^B$ have only radical components by assumption, and since the differential maps on $P^\bullet$ are radical maps, the map $\overline{\beta}$ has only radical components. Since the map $1_{P_B(x)}$ cannot be written as a linear combination of radical maps, it follows that the map $\overline{\alpha}_n$ is given by a matrix whose entry at the fixed position equals $\mu\cdot 1_{P_B(x)}+g$, where $g:P_B(x)\to P_B(x)$ is a radical map. Since $B$ is directed, we have $\dim \operatorname{End}_B(P_B(x))=1$, which implies $g=0$.

Consider now the map $F(\overline{\alpha})$. This is a chain map, which, by construction, is such that the map $F(\overline{\alpha})_n$ is given by a matrix whose entry at the fixed position is $F(\mu \cdot 1_{P_B(x)})=\mu\cdot 1_{P_\Lambda(x)}$.

It follows now, that the map $(\varepsilon-F(\overline{\alpha}))_n$ is given by a matrix, whose entry at the fixed position is
$$f-\mu\cdot 1_{P_\Lambda(x)}=\mu\cdot 1_{P_\Lambda(x)}+f^\prime - \mu\cdot 1_{P_\Lambda(x)}=f^\prime.$$

Repeating this argument, we may write $\varepsilon=\tilde{\varepsilon} + \sum_{k=1}^N F(\overline{\alpha}^k)$, where $\tilde{\varepsilon}$ has only radical components, and where $\sum_{k=1}^NF(\overline{\alpha}^k)\in F(H^B)\subset H^\Lambda$.

Assume now, instead, that the map $\varepsilon_n$ is given by a matrix whose entry at the fixed position equals
$f=\rho_{q^\prime p}+\rho_{r^\prime},$ where $q^\prime p$ is a linear combination of paths in $\Lambda$ and $r^\prime$ is a linear combination of non-trivial paths in $A^{\operatorname{op}}$.
According to the discussion prior to Proposition \ref{proposition:m_2 in ext-algebra}, there is a chain map $\delta:F(P^\bullet)\to F(P^\bullet)[n]$, such that the map $\delta_n:F(P_n)\to F(P_0)$ is given by a matrix, whose only non-zero entry is at the fixed position, where it is $\rho_{r^\prime}$, and with $\delta_k=0$, for $k\neq n$. Now, we argue as in the previous case, and consider the map $(\varepsilon-\delta)_n$, which is given by a matrix whose entry at the fixed position is equal to $\rho_{q^\prime p} + \rho_{r^\prime}-\rho_{r^\prime}=\rho_{q^\prime p}$. It follows that we
 may write $\varepsilon=\hat{\varepsilon}+\delta$, where $\hat{\varepsilon}$ has no component given by a matrix with an entry of the form $\rho_{q^\prime}$, where $q^\prime$ is a linear combination of non-trivial paths in $A^{\operatorname{op}}$, and where $\delta \in \hat{\operatorname{rad}}(\Delta,\Delta)\subset H^\Lambda$.

Next, let $\tilde{\mathcal{D}}^\Lambda$ denote the subspace of maps whose components are matrices whose entries are of the form $\rho_{q^\prime p}$, where $p$ is a linear combination of non-trivial paths in $B$ and $q^\prime$ is some linear combination of paths in $A^{\operatorname{op}}$. The argument above shows that we may write $\mathcal{D}^\Lambda=H^\Lambda + \tilde{\mathcal{D}}^\Lambda$. We observe that $\tilde{\mathcal{D}}^\Lambda$ is a dg subalgebra of $\mathcal{D}^\Lambda$, since, clearly, $\partial^\Lambda (\tilde{\mathcal{D}}^\Lambda)\subset \tilde{\mathcal{D}}^\Lambda$. Letting $\tilde{Z}\subset \tilde{\mathcal{D}}^\Lambda$ denote the subspace spanned by cycles, we may write $\tilde{\mathcal{D}}^\Lambda=\tilde{Z}\oplus \tilde{L}$, where $\tilde{L}$ is some complement. We claim that $\mathcal{D}^\Lambda =Z^\Lambda \oplus \tilde{L}$. We have
$$\mathcal{D}^\Lambda=H^\Lambda + \tilde{\mathcal{D}}^\Lambda =Z^\Lambda + \tilde{\mathcal{D}}^\Lambda =Z^\Lambda +(\tilde{Z}+\tilde{L})=Z^\Lambda+ \tilde{L}, $$
since $H^\Lambda\subset Z^\Lambda$ and $\tilde{Z}\subset Z^\Lambda$. Assume that $f\in Z\cap \tilde{L}$. Then, clearly, $f\in \tilde{Z}$. But, since $\tilde{Z}$ and $\tilde{L}$ are chosen to be complements, we must have $f=0$. 
\end{proof}
\begin{lemma}\label{lemma:projection to ext-algebra is hom-linear}
	
	Let $\varepsilon\in \mathcal{D}^B$ be homogeneous of degree $n$, such that $\varepsilon=F(\overline{\varepsilon})$ for some $\overline{\varepsilon}$. Let $f^\prime\in \hat{\operatorname{rad}}(\Delta,\Delta)$ be a chain map representative of a homomorphism $\Delta_\Lambda(i)\to \Delta_\Lambda(j)$. Then, we have $p^\Lambda(f^\prime \varepsilon)=f^\prime F(p^B (\overline{\varepsilon}))$ as elements of $\operatorname{Ext}_\Lambda^\ast(\Delta,\Delta)$.
\end{lemma}
\begin{proof}
	By definition, we have $p^\Lambda=\chi^\Lambda \pi^\Lambda$, where the map $\pi^\Lambda:\mathcal{D}^\Lambda\to H^\Lambda$ is the natural surjection, and the map $\chi^\Lambda: H^\Lambda\to \operatorname{Ext}_\Lambda(\Delta,\Delta)$ is the isomorphism defined in the proof of Corollary \ref{corollary:proposition:A-inf structure constructions compatible 1}. Similarly, we have $p^B=\chi^B \pi^B$. Consider the following diagram. Write $\mathcal{D}^B=H^B\oplus \operatorname{im}\partial^B\oplus L^B$ and let $\varepsilon=\varepsilon_1+\varepsilon_2+\varepsilon_3\in \mathcal{D}^B$, where $\varepsilon_1\in H^B, \varepsilon_2\in \operatorname{im}\partial^B$ and $\varepsilon_3\in L^B$.
	$$\xymatrix{
	\mathcal{D}^B \ar[d]_-{F}\ar[r]^-{\pi^B} & H^B \ar[r]^-{\chi^B} & \operatorname{Ext}_B^\ast(\mathbb{L},\mathbb{L})\ar[r]^-F & F(\operatorname{Ext}_B^\ast(\mathbb{L},\mathbb{L}) \ar[d]^-{f^\prime \circ \blank}\\
	\mathcal{D}^\Lambda \ar[r]_-{f^\prime \circ \blank}& \mathcal{D}^\Lambda \ar[r]_-{\pi^\Lambda}& H^\Lambda \ar[r]_-{\chi^\Lambda}& \operatorname{Ext}_\Lambda^\ast(\Delta,\Delta)
	}$$
For the top path, we then have
\begin{align*}
	f^\prime F \chi^B \pi^B(\varepsilon)&=f^\prime F \chi^B \pi^B(\varepsilon_1+\varepsilon_2+\varepsilon_3)=f^\prime F\chi^B(\varepsilon_1).
\end{align*}
Consider now the bottom path. The first two arrows correspond to the map
$$\varepsilon_1+\varepsilon_2+\varepsilon_3\mapsto F(\varepsilon_1)+F(\varepsilon_2)+F(\varepsilon_3)\mapsto f^\prime F(\varepsilon_1) +f^\prime F(\varepsilon_2)+ f^\prime F(\varepsilon_3).$$

Arguing as in the proof of Lemma \ref{lemma:form of hom composed with chain map}, we know that the map $f^\prime F(\varepsilon_2)$ has at most one non-zero component, which we denote by $\hat{f} \hat{\varepsilon}_2:P_n\to P_0$. Since $\varepsilon_2\in \operatorname{im}\partial^B$, and the differential maps are radical maps, the map $\varepsilon_2$ is given by a matrix whose entries are of the form $\rho_p$, where $p$ is a linear combination of non-trivial paths in $B$. Moreover, since $f\in \hat{\operatorname{rad}}(\Delta,\Delta)$, the entries of the matrix of $\hat{f}$ are of the form $\rho_{q^\prime}$, where $q^\prime$ is a linear combination of non-trivial paths in $A^{\operatorname{op}}$. Then, $\hat{f}\hat{\varepsilon}_2=0$, since $\rho_q^\prime \rho_p=\rho_{pq^\prime}=0$, according to the dual extension relation.

Since $\varepsilon_3\in L^B$, and we have assumed that $L^B$ may be chosen to consist of maps with only radical maps for components, a similar argument as above shows that $f^\prime F(\varepsilon_3)=0$.

This means that, in the bottom path, we have $f^\prime F(\varepsilon)=f^\prime F(\varepsilon_1)$. 
Recall that we have chosen
	$$H^\Lambda=F(H^B)\oplus (\hat{\operatorname{rad}}(\Delta,\Delta)\circ F(H^B)) \quad \textrm{and} \quad \operatorname{Ext}_\Lambda^\ast(\Delta,\Delta)=F(\operatorname{Ext}_B^\ast(\mathbb{L},\mathbb{L}))\oplus (\operatorname{rad}(\Delta,\Delta)\circ F(\operatorname{Ext}_B^\ast(\mathbb{L},\mathbb{L})).$$
	Then, we get
	\begin{align*}
		\chi^\Lambda \pi^\Lambda (f^\prime F(\varepsilon_1))&=\chi^\Lambda \left[\begin{smallmatrix}
			0 \\ f^\prime F(\varepsilon_1)
		\end{smallmatrix}\right]=\left[\begin{smallmatrix}
		\psi & 0 \\ 0 &\varphi
	\end{smallmatrix}\right]\left[\begin{smallmatrix}
	0 \\ f F(\varepsilon_1)
\end{smallmatrix}\right]=\varphi(fF(\varepsilon_1))=f F\chi^B(\varepsilon_1),
	\end{align*}
so the top path equals the bottom path, and we are done.
\end{proof}
\begin{lemma}\label{lemma:the map h on non-induced part is 0}
	Let $\varepsilon\in \mathcal{D}^\Lambda$ be a homogeneous map of degree $n$, and let $f \in \hat{\operatorname{rad}}(\Delta,\Delta)$. Then, there holds $h^\Lambda(f\circ \varepsilon)=0$.
\end{lemma}
\begin{proof}
	We write $\varepsilon=\varepsilon_1+\varepsilon_2+\varepsilon_3$, where $\varepsilon_1\in H^\Lambda, \varepsilon_2\in \operatorname{im}\partial^\Lambda$ and $\varepsilon_3\in L^\Lambda$. Then, we have $$h^\Lambda(f\varepsilon)=h^\Lambda(f\varepsilon_1)+h^\Lambda(f\varepsilon_2)+h^\Lambda(f\varepsilon_3).$$
	Since we have chosen $H^\Lambda=F(H^B)\oplus \hat{\operatorname{rad}}(\Delta,\Delta)\circ F(H^B)$, we have $f\varepsilon_1\in H^\Lambda$, which implies that $h^\Lambda(f\varepsilon_1)=0$, by definition. Since the components of $\varepsilon_3$ are matrices whose entries are of the form $\rho_{p}$, where $p$ is a linear combination of non-trivial paths in $B$, we have $f\varepsilon_3=0$, by appealing to the dual extension relation, just like in the previous proof. Next, note that $\varepsilon_2\in \operatorname{im}\partial^\Lambda=F(\operatorname{im}\partial^B)\oplus \partial^\Lambda(\hat{L})$, where $\hat{L}$ is as in the statement of Proposition \ref{proposition:A-inf structure constructions compatible}. Since the differential maps on $P^\bullet\to \mathbb{L}$ are radical maps, the space $F(\operatorname{im}\partial^B)$ consists of maps whose components are given by matrices whose entries are of the form $\rho_p$, where $p$ is a linear combination of non-trivial paths in $B$. By Lemma \ref{lemma: L^Lambda consists of only radical maps}, the space $\partial^\Lambda(\hat{L})$ consists of maps whose components are given by matrices whose entries are of the form $\rho_{q^\prime p}$, where $p$ is a linear combination of non-trivial paths in $B$. Then, a similar argument as before, using the dual extension relation, shows that $f\varepsilon_2=0$.
\end{proof}
Now we are ready to prove our main theorem. Note that, by linearity, the following formulae determine the $A_\infty$-structure on $\operatorname{Ext}_\Lambda^\ast(\Delta,\Delta)$ completely.
\begin{theorem}\label{theorem:A-inf structure on ext algebra of standards}
	Let $n\geq 2$. For any $f^\prime_1,\dots, f_n^\prime \in \operatorname{Hom}_\Lambda(\Delta,\Delta)$ and $\varepsilon_1,\dots, \varepsilon_n\in \operatorname{Ext}_B^\ast(\mathbb{L},\mathbb{L})$, such that $\deg \varepsilon_i \geq 1$, for all $1\leq i\leq n$, there hold the following.
	\begin{enumerate}[label=$(\roman*)$]
		\item If there is $1\leq i\leq n-1$, such that $f_i^\prime\in \operatorname{rad}(\Delta_\Lambda(j),\Delta_\Lambda(k))$, we have $m_n^\Lambda(f_n^\prime \varepsilon_n, \dots, f_1^\prime \varepsilon_1)=0.$
		\item We have
	\begin{align*}
		m_n^\Lambda(f_n^\prime \varepsilon_n,\varepsilon_{n-1}, \dots, \varepsilon_1)&=(-1)^{n+1} f_n ^\prime F \left(p^B \varepsilon_n h^B (\lambda^B_{n-1}(\varepsilon_{n-1},\dots, \varepsilon_1))\right).
		\end{align*}
	\end{enumerate}
\end{theorem}
\begin{proof}
	We first prove the following formulae for $\lambda_n^\Lambda$.
	\begin{align*}\left\{ \begin{array}{l l}
		\lambda_n^\Lambda(f_n^\prime \varepsilon_n, \dots, f_1^\prime \varepsilon_1)=0, &\textrm{ if } \exists 1\leq i \leq n-1 : f_i^\prime\in \operatorname{rad}(\Delta_\Lambda(j),\Delta_\Lambda(k));\\
		\lambda_n^\Lambda(f_n^\prime \varepsilon_n,\varepsilon_{n-1}, \dots, \varepsilon_1)=(-1)^{n+1} f_n^\prime \varepsilon_n h^\Lambda \lambda^\Lambda_{n-1}(\varepsilon_{n-1},\dots,\varepsilon_1). & 
	\end{array}\right.\end{align*}
	If $n=2$, the first formula claims that
	$$\lambda_2^\Lambda(f_2^\prime\varepsilon_2, f_1^\prime \varepsilon_1)=f_2^\prime \varepsilon_2 f_1^\prime \varepsilon_1=0.$$
	Using the proof of Proposition \ref{proposition:m_2 in ext-algebra}, we see that a chain map representative of $f_1^\prime$ may be chosen so that the composition $\varepsilon_2 f_1^\prime$ equals the zero chain map. Assume that
	 $$\lambda_k^\Lambda(f^\prime_k \varepsilon_k,\dots, f_1^\prime \varepsilon)=0,\quad \textrm{and}\quad	\lambda_k^\Lambda(f_k^\prime \varepsilon_k,\varepsilon_{k-1}, \dots, \varepsilon_1)=(-1)^{k+1} f_k^\prime \varepsilon_k h^\Lambda \lambda^\Lambda_{k-1}(\varepsilon_{k-1},\dots,\varepsilon_1),$$ for all $k<n$, and assume that $f_i^\prime\in \operatorname{rad}(\Delta,\Delta)$, for some $2\leq i\leq n-1$. Consider the sum
	$$\lambda_n^\Lambda(f_n^\prime\varepsilon,\dots, f_1^\prime \varepsilon_1)=\sum_{r+s=n} (-1)^{s+1}\lambda_2(h^\Lambda\lambda_r^\Lambda(f_n^\prime \varepsilon_n, \dots, f_{s+1}^\prime\varepsilon_{s+1})\otimes h^\Lambda \lambda_s(f_s^\prime \varepsilon_s,\dots, f_1^\prime \varepsilon_1)),\quad s\geq 1.$$
	By the first claim of the induction hypothesis, all terms in the sum vanish, except for
	\begin{align*}(-1)^{n-i+1} \lambda^\Lambda_2( h\lambda^\Lambda_{n-i+1}(f_n^\prime\varepsilon_n,\dots, f_{i+1}^\prime \varepsilon_{i+1})\otimes h\lambda^\Lambda_i(f^\prime _i \varepsilon_i,\dots, f_1^\prime\varepsilon_1)).\end{align*}
	By the second claim of the induction hypothesis and Lemma \ref{lemma:the map h on non-induced part is 0}, we have
	\begin{align*}h^\Lambda\lambda^\Lambda_i(f_i^\prime \varepsilon_i,\dots, f_1^\prime \varepsilon_1)&=(-1)^{i+1}h^\Lambda (f_i^\prime \varepsilon_i h^\Lambda \lambda^\Lambda_{i-1}(f_{i-1}^\prime\varepsilon_{i-1},\dots, f_1^\prime\varepsilon_1))=0.
	\end{align*}
    Note that $\lambda_n^\Lambda(f_n^\prime \varepsilon_n,\dots,f_1^\prime \varepsilon_1)=0$ implies that $m_n^\Lambda(f_n^\prime\varepsilon_n,\dots, f_1^\prime \varepsilon_1)=0$.
	Now we prove the second formula. If $n=2$, the second formula claims that
	$$\lambda_2^\Lambda(f_2^\prime \varepsilon_2, \varepsilon_1)=- f_2^\prime \varepsilon_2 h^\Lambda \lambda_1^\Lambda (\varepsilon_1).$$
	This is true, as $\lambda_2^\Lambda$ is just composition, and $h^\Lambda \lambda_1^\Lambda=-\operatorname{id}$, by definition. Suppose
	$$	\lambda_k^\Lambda(f_k^\prime \varepsilon_k,\varepsilon_{k-1}, \dots, \varepsilon_1)=(-1)^{k+1} f_k^\prime \varepsilon_k h^\Lambda \lambda^\Lambda_{k-1}(\varepsilon_{k-1},\dots,\varepsilon_1),$$
	for all $k<n$. Consider the sum
	\begin{align*}\lambda^\Lambda_n(f_n^\prime\varepsilon_n,\dots, \varepsilon_1)=\sum_{r+s=n}(-1)^{s+1} \lambda_2^\Lambda(h^\Lambda \lambda^\Lambda_r(f_n^\prime \varepsilon_n,\dots, \varepsilon_{s+1})\otimes h^\Lambda \lambda^\Lambda_s(\varepsilon_s,\dots, \varepsilon_1)).\end{align*}
	If $r>1$, we have
	\begin{align*}
		h^\Lambda \lambda_r^\Lambda(f_n^\prime \varepsilon_n,\dots, \varepsilon_{s+1})&=h^\Lambda \left((-1)^{r+1} f_n^\prime \varepsilon_n h^\Lambda \lambda^\Lambda_{k-1}(\varepsilon_{k-1},\dots,\varepsilon_1)\right)=(-1)^{r+1}h^\Lambda\left(f_n^\prime \varepsilon_n h^\Lambda \lambda_{k-1}^\Lambda(\varepsilon_{k-1},\dots,\varepsilon_1)\right)=0,
	\end{align*}
	again using the induction hypothesis and Lemma \ref{lemma:the map h on non-induced part is 0}. Therefore, all terms in the sum with $r>1$ vanish. It follows, then, that we have
	\begin{align*}
		\lambda^\Lambda_n(f_n^\prime\varepsilon_n,\dots, \varepsilon_1)&=(-1)^n \lambda^\Lambda_2 (-f_n^\prime \varepsilon_n \otimes h^\Lambda \lambda^\Lambda_{n-1}(\varepsilon_{n-1},\dots, \varepsilon_1))=(-1)^{n+1} f_n^\prime \varepsilon_n h^\Lambda \lambda^\Lambda_{n-1}(\varepsilon_{n-1},\dots,\varepsilon_1).
	\end{align*}
Plugging this formula into the definition of $m_n^\Lambda$, we get
	\begin{align*}
		m_n^\Lambda(f_n^\prime \varepsilon_n,\dots, \varepsilon_1)&=p^\Lambda \lambda^\Lambda_n {(i^\Lambda)}^{\otimes n}(f_n^\prime \varepsilon_n,\dots,  \varepsilon_1)=p^\Lambda (-1)^{n+1} f_n^\prime \varepsilon_n h^\Lambda \lambda^\Lambda_{n-1}(\varepsilon_{n-1},\dots,\varepsilon_1)\\
		&=(-1)^{n+1}p^\Lambda \left(f_n^\prime \varepsilon_n h^\Lambda \lambda^\Lambda_{n-1}(\varepsilon_{n-1},\dots,\varepsilon_1)\right)\\
		&=(-1)^{n+1}p^\Lambda\left( f_n^\prime \varepsilon_n h^\Lambda F(\lambda_{n-1}^B(\varepsilon_{n-1},\dots,\varepsilon_1))\right)\\
		&=(-1)^{n+1} p^\Lambda \left( f_n^\prime \varepsilon_n F(h^B)F(\lambda_{n-1}^B(\varepsilon_{n-1},\dots, \varepsilon_1))\right)\\
		&=(-1)^{n+1}p^\Lambda\left( f_n^\prime \varepsilon_n F(h^B\lambda_{n-1}^B(\varepsilon_{n-1},\dots, \varepsilon_1))\right)\\
		&=(-1)^{n+1}f_n^\prime F(p^B \varepsilon_n h^B(\lambda_{n-1}^B(\varepsilon_{n-1},\dots,\varepsilon_1))),
	\end{align*}
	by applying Corollary \ref{corollary:proposition:A-inf structure constructions compatible 1}, Corollary \ref{corollary:proposition:A-inf structure constructions compatible 2} and Lemma \ref{lemma:projection to ext-algebra is hom-linear}. \qedhere
\end{proof}

\section{$A_\infty$-structure on the $\operatorname{Ext}$-algebra of simple modules over $\faktor{K\mathbb{A}_n}{(\operatorname{rad}K\mathbb{A}_n)^\ell}$}
Next, we want to apply Theorem \ref{theorem:A-inf structure on ext algebra of standards} in an example. Since the theorem describes the $A_\infty$-multiplications on $\operatorname{Ext}_\Lambda^\ast(\Delta,\Delta)$ in terms of the data consituting the $A_\infty$-structure on $\operatorname{Ext}_B^\ast(\mathbb{L},\mathbb{L})$, any example requires that we first compute the $A_\infty$-structure on $\operatorname{Ext}_B^\ast(\mathbb{L},\mathbb{L})$.
\subsection{Quiver and relations} We consider the algebra $B=\faktor{K\mathbb{A}_n}{(\operatorname{rad}K\mathbb{A}_n)^\ell}$, for $\ell\geq 3$. Then, $B$ is described by the quiver
$\xymatrix@=0.3cm{
1 \ar[r]^-{\alpha_1}  & \dots \ar[r]^-{\alpha_{n-1}}& n
}$ modulo the relations $\alpha_{i+\ell}\dots \alpha_{i}=0$. The Loewy diagrams of the indecomposable projective $B$-modules are:
$$P_B(i):\vcenter{\xymatrix@=0.3cm{
	i \ar[d] \\ \vdots  \ar[d] \\ i+\ell-1	
}},\quad \textrm{if}\quad i\leq n-\ell, \quad \textrm{and} \quad P_B(i): \vcenter{\xymatrix@=0.3cm{i\ar[d]\\ \vdots \ar[d] \\n }}, \quad \textrm{if}\quad i>n-\ell.$$
It is easy to compute that there is a minimal projective resolution $P^\bullet\to L_B(i)$, which has terms
\begin{align*}P^k=\left\{\begin{array}{l l}
	P_B(i+q\ell), & \textrm{if}\quad k=2q; \\
	P_B(i+q\ell +1), & \textrm{if}\quad k=2q+1.
\end{array}\right.\end{align*}

For indecomposable projective $B$-modules, $P_B(i)$ and $P_B(j)$, such that $j\geq i$, and $|i-j|<\ell$, we have $\dim \operatorname{Hom}_B(P_B(j),P_B(i))=1$. This space contains scalar multiples of the map $\rho_\alpha$, that is, right multiplication with the (unique) path $\alpha:i\to j$ in the quiver. Denote this map by $f_j^i$.

Now, we check that, in the decomposition $\mathcal{D}^B=H^B\oplus \operatorname{im}\partial^B \oplus L^B$, the space $L^B$ may be chosen in such a way that any component of a map $\varepsilon\in L^B$ has components given by matrices whose entries are of the form $\rho_p$, where $p$ is a linear combination of non-trivial paths in $B$. To this end, let $\varepsilon\in \mathcal{D}^B$ be a map, homogeneous of degree $k$, such that the matrix of the map $\varepsilon_k$ has an entry, $g$, at position $(r,s)$, which is not a radical map. Since the maps $f_j^i$ are radical for $i\neq j$, the map $g$ is an endomorphism of some indecomposable projective $B$-module, $P_B(x)$. Since $B$ is directed, we have $\dim \operatorname{End}_B(P_B(x))=1$, so it follows that $g=\mu\cdot 1_{P_B(x)}$, for some scalar $\mu\in K$.

$$\xymatrix{
P_{k+1} \ar@{-->}[d]\ar[r] & P_k \ar[r] \ar[d]^-{\varepsilon_k} & P_{k-1} \\
P_1 \ar[r] & P_0 \ar[r] & 0	
}$$
Note that each module $P_k$ has exactly $n$ non-isomorphic summands, since each term in the minimal projective resolution of a single simple $B$-module is exactly one indecomposable projective $B$-module. Therefore, consider the following picture.

$$\xymatrix{
P_B(y) \ar@{-->}[d]^-{f_y^z} \ar[r]^-{f_y^x} &P_B(x) \ar[r] \ar[d]^-{1_{P_B(x)}} & \dots \\
P_B(z) \ar[r]_-{f_z^x} & P_B(x) \ar[r] & 0
}$$
It is clear that, if we let the dashed arrow represent the map $f_y^z$, we get a commutative square, since $f_y^x=f_z^x f_y^z$. Now, putting the map $f_y^z$ into a matrix in the appropriate position, it is clear that we can make the square

$$\xymatrix{
	P_{k+1} \ar@{-->}[d]\ar[r] & P_k \ar[r] \ar[d]^-{\varepsilon_k} & P_{k-1} \\
	P_1 \ar[r] & P_0 \ar[r] & 0	
}$$
commute. Continuing in this way, we obtain a chain map, $\delta$, such that the matrix of $\delta_k:P_k\to P_0$ has an entry $1_{P_B(x)}$ at position $(r,s)$.

Note that this procedure does not work in general, but crucially depends on the form of the minimal projective resolution of simple $B$-modules and, indeed, on the quiver of $B$. Our procedure guarantees that the map $\varepsilon-\mu\cdot \delta$ is such that the component $(\varepsilon-\mu\cdot \delta)_k$ is given by a matrix whose entry at position $(r,s)$ is 0. Repeating our argument, it follows that we may write $\varepsilon=\tilde{\varepsilon}+\gamma$, where $\gamma\in H^B$ and where $\tilde{\varepsilon}$ has only radical components. Now, we argue exactly as in the proof of Lemma \ref{lemma: L^Lambda consists of only radical maps}, to see that this implies that $L^B$ may be chosen to be of the desired form.

To compute $\operatorname{Ext}^m(L_B(i),L_B(j))$, we apply $\operatorname{Hom}_B(\blank, L_B(j))$ to  $P^\bullet$, obtaining the complex:
$$\xymatrix@=0.3cm{
	0\ar[r] & \operatorname{Hom}_B(P_B(i),L_B(j)) \ar[r] & \operatorname{Hom}_B(P_B(i+1),L_B(j)) \ar[r] & \operatorname{Hom}_B(P_B(i+\ell),L_B(j)) \ar[r] & \dots
}$$
We know that $\dim \operatorname{Hom}_B(P_B(i),L_B(j))=\delta_{ij}$, with the space $\operatorname{Hom}_B(P_B(i),L_B(i))$ containing (up to scalar) only the projection $P_B(i)\to \operatorname{top}P_B(i)$, if $i=j$, and the zero map otherwise. Because the differential maps on $P^\bullet$ are radical maps, we get $\operatorname{Ext}_B^m(L_B(i),L_B(j))=\operatorname{Hom}_B(P^m, L_B(j))$, which implies that
\begin{align*}
	\dim \operatorname{Ext}^{2k}_B(L_B(i),L_B(j))=\left\{\begin{array}{c l}
		1, & \textrm{if}\quad j=i+k\ell;\\
		0, & \textrm{otherwise},
	\end{array}\right.\end{align*}
and
\begin{align*}\dim \operatorname{Ext}^{2k+1}_B(L_B(i),L_B(j))=\left\{\begin{array}{c l}
		1, & \textrm{if}\quad j=i+k\ell+1;\\
		0, & \textrm{otherwise}.
	\end{array}\right.\end{align*}
These extensions have chain map representatives
\begin{align*}\xymatrix@=0.3cm{
	\dots\ar[r]&P_B(i+(k+1)\ell)\ar[r]\ar[d]^-1&P_B(i+k\ell+1)\ar[r]\ar[d]^-1&P_B(i+k\ell)\ar[d]^-1\ar[r]&\dots\ar[r]&P_B(i) \\
	\dots\ar[r]&P_B(i+(k+1)\ell) \ar[r]&P_B(i+k\ell+1)\ar[r] &P_B(i+k\ell),
}\end{align*}
and
\begin{align*}\xymatrix@=0.3cm{
	\dots\ar[r]&P_B(i+(k+1)\ell+1)\ar[r]\ar[d]^-1&P_B(i+(k+1)\ell)\ar[r]\ar[d]^-f&P_B(i+k\ell+1)\ar[r]\ar[d]^-1&P_B(i+k\ell)\ar[r]&\dots\ar[r]&P_B(i) \\
	\dots\ar[r]&P_B(i+(k+1)\ell+1)\ar[r]&P_B(i+k\ell+2) \ar[r]&P_B(i+k\ell+1),
}\end{align*}
respectively. Let $\varepsilon_i$ be the natural basis vector of the space $\operatorname{Ext}_B^1(L_B(i),L_B(i+1))$, and let $\delta_i$ be the natural basis vector of the space $\operatorname{Ext}_B^2(L_B(i),L_B(i+\ell))$. It is then easy to check that we must have $\varepsilon_{i+\ell}\delta_i=\delta_{i+1}\varepsilon_i$, as elements of $\operatorname{Ext}_B^3(L_B(i), L_B(i+\ell+1))$. Since $\operatorname{Ext}^2_B(L_B(i),L_B(i+2))=0$, and $\ell\geq 3$, we must have $\varepsilon_{i+1}\varepsilon_i=0$ for all $i$.

Let $C$ be the path algebra of the quiver $Q$, given by:
\begin{align*}\xymatrix{
	1\ar[r]^-{\alpha_1} \ar@/_1.5pc/[rrr]_-{\beta_1} & 2 \ar@/_1.5pc/[rrr]_-{\beta_2} \ar[r]^-{\alpha_2} & \dots \ar[r]& 1+\ell\ar[r]^-{\alpha_{1+\ell}} \ar@/_1.5pc/[rrrr]_-{\beta_{1+\ell}}& 2+\ell \ar[r]&\dots \ar[r]& 2\ell \ar[r]^-{\alpha_{2\ell}} & 1+2\ell \ar[r]&\dots \ar[r]& n
},\end{align*}
modulo the relations $\alpha_{i+1}\alpha_i=0$, and $\alpha_{i+\ell}\beta_i=\beta_{i+1}\alpha_i$. The above observations imply that there exists a homomorphism of algebras $\Phi:C\to \operatorname{Ext}_B^\ast(\mathbb{L},\mathbb{L})$, defined by $\alpha_i\mapsto \varepsilon_i$ and $\beta_i\mapsto \delta_i$.

Consider now the projective resolution $P^\bullet\to L_B(i)$. If we let $m_{j,k}$ denote the multiplicity of $P_B(j)$ in the $k$th term $P^k$, we have $\dim \operatorname{Ext}_B^k(L_B(i),L_B(j))=\dim \operatorname{Hom}_B(P^k,L_B(j))=m_{j,k}.$ From the form of the projective resolution of $L_B(i)$, we know that

\begin{align*}P^k=\left\{\begin{array}{l l}
	P_B(i+q\ell), & \textrm{if}\quad k=2q; \\
	P_B(i+q\ell +1), & \textrm{if}\quad k=2q+1,
\end{array}\right.\implies m_{j,k}=\left\{\begin{array}{l l}
1, & \textrm{if}\quad k=i+q\ell; \\
1, & \textrm{if}\quad k=i+q\ell+1;\\
0, & \textrm{otherwise}.
\end{array}\right.\end{align*}
Here, $q$ must be such that $i+q\ell \leq n$ and $i+q\ell+1 \leq n$, respectively. Next, we claim that $$\delta_{i+(q-1)\ell}\dots\delta_{i+\ell}\delta_i,\quad \textrm{and},\quad \alpha_{i+q\ell}\delta_{i+(q-1)\ell}\dots\delta_{i+\ell}\delta_i,$$ are non-split extensions. To see this, we look at the chosen chain map representatives of and note that they have components which are the identity homomorphism on some projective. This ensures the chain maps are not null-homotopic. Therefore, the extensions $\varepsilon_i$ and $\delta_i$ generate $\operatorname{Ext}_B^\ast(\mathbb{L},\mathbb{L})$, so that $\Phi$ is surjective. Finally, we do a dimension count. We have $\dim \operatorname{Ext}_B^k(L_B(i),L_B(j))=m_{j,k}$, which we compare to the dimension of the degree $k$ part of the space $e_j C e_i$.

In $C$, our relations imply that
\begin{align*}\beta_{i+r\ell+1}\dots \beta_{i+\ell+1}\beta_{i+1}\alpha_i=\beta_{i+r\ell+1}\dots \beta_{i+\ell+1}\alpha_{i+\ell}\beta_i=\dots=\alpha_{r\ell}\beta_{(r-1)\ell}\dots\beta_{i+\ell}\beta_i\end{align*}

which shows that non-zero paths in $C$ have one of the two following forms.
\begin{enumerate}
	\item The path is of the form $\beta_{i+(r-1)\ell} \dots \beta_{i+\ell}\beta_i$ going from $i$ to $i+r\ell$.
	\item The path is of the form $\alpha_{i+r\ell}\beta_{i+(r-1)\ell}\dots \beta_{i+\ell}\beta_i$ going from $i$ to $i+r\ell+1$.
\end{enumerate}
This shows that 
$$\dim \operatorname{Ext}_B^k(L_B(i),L_B(j))=m_{j,k}=\left\{\begin{array}{l l}
	1, & \textrm{if}\quad k=i+q\ell; \\
	1, & \textrm{if}\quad k=i+q\ell+1;\\
	0, & \textrm{otherwise} \end{array}\right.=\dim e_j C_k e_i,$$
which implies that $\dim \operatorname{Ext}_B^\ast(\mathbb{L},\mathbb{L})=\dim C$, so that $\Phi$ is an isomorphism of algebras.

\subsection{$A_\infty$-structure} In \cite{Madsen}, Madsen computed the $A_\infty$-structure on the $\operatorname{Ext}$-algebra of simple modules over the path algebra of quiver
\begin{align*}
	\xymatrix{
		1 \ar@(dr,ur)_-{\alpha} }
\end{align*} modulo the relation $\alpha^n=0, \ n\geq 3$ as an example. We use his computation to predict the formula for the present case.  We abuse notation and suppress the indices on arrows $\beta$ and $\alpha$ when writing their concatenations. We write a basis element of $\operatorname{Ext}_B^\ast(\mathbb{L},\mathbb{L})$ as $\alpha^x\beta^y$. Here, $x\in \{0,1\}$, and $y$ is such that if the starting vertex of the first $\beta$ is $s$, then $s+y\ell \leq n$. Consider $m_k(\alpha^{x_k}\beta^{y_k},\dots, \alpha^{x_1}\beta^{y_1})$ for $k\geq 3$. This produces an extension from some $L_B(s)$ to $L_B(t)$ where $t=s+\sum\limits_{i=1}^k x_i + y_i\ell.$

We have seen above, that there are non-zero extensions from $L_B(s)$ to $L_B(t)$ if and only if
\begin{align*}t-s \equiv 0 \mod \ell\quad  \textrm{or}\quad t-s\equiv 1 \mod \ell \iff \sum\limits_{i=1}^k x_i +\ell y_i\equiv 0 \mod \ell \quad \textrm{or}\quad \sum\limits_{i=1}^k x_i+\ell y_i\equiv 1 \mod \ell.\end{align*}
Suppose $k\neq \ell$.
The combined degree of the arguments of $m_k$ is $\sum\limits_{i=1}^k x_i + 2y_i$. Since $m_k$ is of degree $2-k$, this should produce an extension of degree $2-k+ \sum\limits_{i=1}^k x_i + 2y_i$.

\begin{enumerate}[label=$(\roman*)$]
	\item If $\sum\limits_{i=1}^k x_i=q\ell$, the non-split extension from $L_B(s)$ to $L_B(t)$  is of degree $2\big(q+\sum\limits_{i=1}^k y_i\big)$. Then,
	\begin{align*}
		2\big(q+\sum\limits_{i=1}^ky_i\big)=2-k+q\ell +2\sum_{i=1}^ky_i \implies (2-\ell)q=2-k \implies q=\frac{2-k}{2-\ell}.
	\end{align*}
If $k<\ell$, we get $q>1$. This is a contradiction, since $0\leq \sum_{i=1}^k x_i \leq k$. If $k>\ell$, we have $q<1$, which implies $q=0$, since $q$ is a non-negative integer. This implies the equality
$$2-k+\sum_{i=1}^k 2y_i=2\sum_{i=1}^ky_i,$$
which implies $2=k$. Since $k\geq 3$, by assumption, this is a contradiction.
\item If $\sum\limits_{i=1}^k x_i=q\ell+1$, an identical argument works.
\end{enumerate}
Left to consider is the case $k=\ell$. There are three possibilities for the condition on the sum $\sum\limits_{i=1}^k x_i+ \ell y_i$.
\begin{enumerate}[label=$(\roman*)$]
	\item We have $x_i=0$ for all $1\leq i \leq \ell$. The sum of degrees of the arguments is $2\sum\limits_{i=1}^\ell y_i$. Since $m_\ell$ is of degree $2-\ell$ this would yield an extension of degree $2-\ell+ 2\sum\limits_{j=1}^\ell y_j$ from $L_B(s)$ to $L_B(t)$ but the only such extension is of degree $2\sum\limits_{j=1}^\ell y_j.$ This is contradiction since $\ell\geq 3$.
	\item We have $x_i=1$ for exactly one $1\leq i\leq \ell$. We mimic the previous case.
	\item We have $x_i=1$ for all $1\leq i \leq \ell$. Then we are dealing with an expression of the form $m_\ell(\alpha \beta^{y_\ell},\dots,\alpha \beta^{y_1})$.
\end{enumerate}
We claim that $m_\ell(\alpha \beta^{y_\ell},\dots,\alpha \beta^{y_1})=\beta^{y_{\ell}+1}\dots \beta^{y_1}$. Consider $\gamma=\lambda_r(\alpha\beta^{y_r},\dots, \alpha\beta^{y_1})$ for $r<\ell$. Put $\sigma=\sum\limits_{i=1}^r y_j$. Suppose the starting vertex of the first $\beta$ is $s$. Then $\gamma$ is a chain map from the projective resolution of $L_B(s)$ to that of $L_B(t)$, where $t=s+r+\ell \sigma.$ We see that $\deg \gamma=2-r + r+2\sigma=2(1+\sigma).$

Earlier we saw that

\begin{align*}P^k=\left\{\begin{array}{l l}
	P_B(i+q\ell) & \textrm{if}\quad k=2q \\
	P_B(i+q\ell +1) & \textrm{if}\quad k=2q+1
\end{array}\right.\end{align*}

which implies that $P^{\deg \gamma}=P_B\left(s+ \ell(1+\sigma)\right).$ We claim that $\lambda_r(\alpha\beta^{y_r},\dots, \alpha\beta^{y_1})$ is the chain map

\begin{align*}\xymatrix@=0.3cm{
	\dots\ar[r]&	P_B(s+l(1+\sigma)+1) \ar[r] \ar[d]^-{f}& P_B(s+\ell(1+\sigma))\ar[d]^-{f} \ar[r] &\dots \ar[r]&P_B(s)	\\
	\dots\ar[r]& P_B(s+r+\ell\sigma+1) \ar[r] & P_B(s+r+\ell\sigma)
}\end{align*}
and that it is the image under the differential of the map
$$\xymatrix@=0.3cm{
	\dots \ar[r]&	P_B(s+\ell(2+\sigma)\ar[d]^-f\ar[r]&	P_B(s+\ell(1+\sigma)+1) \ar[r] \ar[d]^-{0}& P_B(s+\ell(1+\sigma))\ar[d]^-{f} \ar[r] &\dots \ar[r]&P_B(s)	\\
	\dots\ar[r]&	P_B(s+r+(\ell+1)\sigma+1)\ar[r]& P_B(s+r+(\ell+1)\sigma) \ar[r] & P_B(s+r+\ell\sigma+1)
}$$
We proceed by induction. It is clear that the claim holds for $r=2$, by just writing down the composition. Moreover, it is clear that $\gamma$ is the image under the differential of a map of the form given above, which follows from the fact that if we have vertices $a,b,c$ such that $a\leq b\leq c$ and $|c-a|<\ell$, then $f_a^c=f_b^c f_a^b.$

Assume our claim holds for all $2\leq r,s<k$. We put $\delta=\lambda_r(\alpha\beta^{y_r},\dots, \alpha \beta^{y_{s+1}})$ and $\gamma=\lambda_s(\alpha \beta^{y_s},\dots, \alpha\beta^{y_1}).$ Let $P^\bullet,Q^\bullet$ and $R^\bullet$ denote minimal projective resolutions of $L_B(a),L_B(b)$ and $L_B(c)$, respectively. We have
\begin{align*}\deg \gamma=2\big(1+\sum_{i=1}^s y_j\big)\coloneqq d_\gamma \quad\textrm{and}\quad \deg\delta=2\big(1+\sum_{i=s+1}^k y_j\big)\coloneqq d_\delta.\end{align*}
Then $\gamma$ is the chain map
\begin{align*}\gamma=\xymatrix@=0.3cm{
	\dots\ar[r]& P^{d_\gamma+2}\ar[r]\ar[d]^-f&P^{d_\gamma+1}\ar[r]\ar[d]^-f&P^{d_\gamma}\ar[d]^-f \ar[r]& \dots \ar[r] & P^0	\\
	\dots\ar[r]&Q^2\ar[r]& Q^1\ar[r]& Q^0
},\quad\textrm{and}\quad \delta=\xymatrix@=0.3cm{
\dots\ar[r]& Q^{d_\delta+2}\ar[r]\ar[d]^-f&Q^{d_\delta+1}\ar[r]\ar[d]^-f&Q^{d_\delta}\ar[d]^-f \ar[r]& \dots \ar[r] & Q^0	\\
\dots\ar[r]&R^2\ar[r]& R^1\ar[r]& R^0
}\end{align*}
Also by assumption, we have
\begin{align*}h\gamma=\xymatrix@=0.3cm{
	\dots\ar[r]& P^{d_\gamma+2}\ar[r]\ar[d]^-f&P^{d_\gamma+1}\ar[r]\ar[d]^-0&P^{d_\gamma}\ar[d]^-f \ar[r]& \dots \ar[r] & P^0	\\
	\dots\ar[r]&Q^3\ar[r]& Q^2\ar[r]& Q^1
}\quad\textrm{and}\quad h\delta=\xymatrix@=0.3cm{
\dots\ar[r]& Q^{d_\delta+2}\ar[r]\ar[d]^-f&Q^{d_\delta+1}\ar[r]\ar[d]^-0&Q^{d_\delta}\ar[d]^-f \ar[r]& \dots \ar[r] & Q^0	\\
\dots \ar[r]&R^3\ar[r]& R^2\ar[r]& R^1
}\end{align*}
and the composition of these is the map
\begin{align*}\xymatrix@=0.3cm{
	\dots\ar[r]&P^{d_\gamma+d_\delta+1}\ar[r]\ar[d]^-0&P^{d_\gamma+d_\delta} \ar[d]^-f\ar[r]&P^{d_\gamma+d_\delta-1}\ar[d]^-0\ar[r]&\dots\ar[r]&P^{d_\gamma+2}\ar[r]\ar[d]^-f&P^{d_\gamma+1}\ar[r]\ar[d]^-0&P^{d_\gamma}\ar[d]^-f \ar[r]& \dots \ar[r] & P^0	\\
	\dots\ar[r]&P^{d_\delta+2}\ar[r]\ar[d]^-f&Q^{d_\delta+1}\ar[d]^-0 \ar[r]& Q^{d_\delta}\ar[d]^-f \ar[r]&\dots\ar[r]&Q^3\ar[r]& Q^2\ar[r]& Q^1 \\
	\dots\ar[r]&R^3\ar[r]& R^2\ar[r]& R^1 \ar[r]& R^0
}\end{align*}
which is 0. Left to consider are the cases $s=1$ and $r=1$. If $s=1$ we have the composition
\begin{align*}\xymatrix@=0.3cm{
	P^{2y_1+d_\delta+3}\ar[d]^-1 \ar[r] &P^{2y_1+d_\delta+2}\ar[d]^-f \ar[r] &P^{2y_1+d_\delta+1}\ar[d]^-1 \ar[r]&\dots\ar[r]&P^{2y_2+2}\ar[r] \ar[d]^-f&P^{2y_1+1} \ar[d]^-1 \ar[r] &\dots \ar[r]& P^0	\\
	Q^{d_\delta+2}\ar[d]^-f\ar[r]&Q^{d_\delta+1}\ar[d]^-0\ar[r]&Q^{d_\delta}\ar[d]^-f\ar[r]&\dots\ar[r]&Q_1 \ar[r]&Q_0 \\
	R^3\ar[r]&R^2\ar[r]& R^1
}\end{align*}
and if $r=1$ we have
\begin{align*}\xymatrix@=0.3cm{
	\dots\ar[r]&P^{2y_r+d_\gamma+2}\ar[r] \ar[d]^-f&P^{2y_r+d_\gamma+1}\ar[r] \ar[d]^-{0}& P^{2y_r+d_\gamma}\ar[d]^-f\ar[r]&\dots\ar[r]& P^{d_\gamma+1}\ar[r]\ar[d]^-0&P^{d_\gamma}\ar[d]^-f \ar[r]& \dots \ar[r] & P^0	\\
	\dots\ar[r]&Q^{2y_r+3} \ar[r] \ar[d]^-1&Q^{2y_r+2}\ar[r] \ar[d]^-f& Q^{2y_r+1}\ar[d]^-1\ar[r]&\dots\ar[r]& Q^2\ar[r]& Q^1 \ar[r] & Q^0\\
	\dots\ar[r]&R^2\ar[r]&R^1 \ar[r] & R^0
}\end{align*}
and adding the two maps proves the claim. Finally, we consider the case $k=\ell$. Put $\omega=\lambda_\ell(\alpha\beta^{y_\ell},\dots, \alpha\beta^{y_1})$. Then $\omega$ is a chain map from the projective resolution of $L_B(i)$ to that of $L_B(i+t)$, where $t=(1+\sigma)\ell$. It suffices to check that, in this case, the projective resolutions line up in the following way.

\begin{align*}\xymatrix@=0.3cm{
	\dots \ar[r] & P_B(i+t+1)\ar[r] \ar[d]^-{\omega_{i+t+1}}& P_B(i+t) \ar[d]^-{\omega_{i+t}}	\ar[r] & \dots \ar[r] &P_B(i) \\
	\dots \ar[r] & P_B(i+t+1) \ar[r] & P_B(i+t)
}\end{align*}
Our claim now implies that $\omega_j=\operatorname{id}_{P_B(j)}$ for all $j\geq i+t$, so that
\begin{align*}m_\ell (\alpha\beta^{y_\ell},\dots, \alpha\beta^{y_1})=\beta^{\sum\limits_{i=1}^\ell y_i +1}.\end{align*}
We remark that this formula could also be obtained by applying \cite[Theorem~4.9]{Tamaroff2021}.
\section{Application to the dual extension algebra}
Having computed the $A_\infty$-structure on $\operatorname{Ext}_B^\ast(\mathbb{L},\mathbb{L})$, we turn to $\operatorname{Ext}_\Lambda^\ast(\Delta,\Delta)$ for $\Lambda=\mathcal{A}(B,B^{\operatorname{op}})$. We have $\operatorname{Ext}_\Lambda^\ast(\Delta,\Delta)\cong\mathcal{A}(\operatorname{Ext}_B^\ast(\mathbb{L},\mathbb{L}),B)$ so $\operatorname{Ext}_\Lambda^\ast(\Delta,\Delta)$ is isomorphic to the path algebra of the quiver

\begin{align*}\xymatrix{
	1\ar@<-0.5ex>[r]_{a_1} \ar@<.5ex>[r]^-{\alpha_1} \ar@/_2pc/[rrr]_-{\beta_1} & 2 \ar@/_2pc/[rrr]_-{\beta_2} \ar@<-0.5ex>[r]_{a_2} \ar@<.5ex>[r]^-{\alpha_2} & \dots \ar@<-0.5ex>[r] \ar@<.5ex>[r]& 1+\ell\ar@<-0.5ex>[r]_{a_{1+\ell}} \ar@<.5ex>[r]^-{\alpha_{1+\ell}} \ar@/_2pc/[rrrr]_-{\beta_{1+\ell}}& 2+\ell \ar@<-0.5ex>[r] \ar@<.5ex>[r]&\dots \ar@<-0.5ex>[r] \ar@<.5ex>[r]& 2\ell \ar@<-0.5ex>[r]_-{a_{2\ell}} \ar@<.5ex>[r]^-{\alpha_{2\ell}} & 1+2\ell \ar@<-0.5ex>[r] \ar@<.5ex>[r]&\dots \ar@<-0.5ex>[r] \ar@<.5ex>[r]& n
}\end{align*}
modulo the relations
\begin{align*}\alpha_{i+1}\alpha_i=0,\quad \alpha_{i+\ell}\beta_i=\beta_{i+1}\alpha_i\quad \textrm{and}\quad a_{i+\ell-1}\dots a_{i+1}a_i=0\end{align*}
as well as the dual extension relations, $\alpha_{i+1}a_i=0$ and $\beta_{i+1}a_i=0$.

In light of Theorem \ref{theorem:A-inf structure on ext algebra of standards}, we consider an expression of the form $m_n(g^\prime_n \varepsilon_n,\dots, \varepsilon_1)$
where $g_n^\prime \in \operatorname{Hom}_\Lambda(\Delta,\Delta)$ and $\varepsilon_1,\dots,\varepsilon_n\in \operatorname{Ext}_B^\ast(\mathbb{L},\mathbb{L})$. By our theorem, there holds
\begin{align*}m_n(g^\prime_n \varepsilon_n,\dots, \varepsilon_1)=(-1)^{n+1} g_n ^\prime F \left(p^B \varepsilon_n h^B (\lambda^B_{n-1}(\varepsilon_{n-1},\dots, \varepsilon_1))\right).\end{align*}
We have non-split extensions $\beta^k\in \operatorname{Ext}_\Lambda^{2k}(\Delta_\Lambda(i),\Delta_\Lambda(i+k\ell))$ and $\alpha\beta^k\in \operatorname{Ext}_\Lambda^{2k+1}(\Delta_\Lambda(i),\Delta_\Lambda(i+k\ell+1))$
which are induced from the corresponding extensions between simple $B$-modules. These can be composed with $g^\prime \in \operatorname{Hom}_\Lambda(\Delta,\Delta)$ to obtain new extensions
\begin{align*}g^\prime\beta^k\in \operatorname{Ext}_\Lambda^{2k}(\Delta_\Lambda(i),\Delta_\Lambda(j))\quad \textrm{and}\quad g^\prime\alpha\beta^k\in \operatorname{Ext}_\Lambda^{2k+1}(\Delta_\Lambda(i),\Delta_\Lambda(j+1)).\end{align*}
If $\lambda^B_{n-1}(\varepsilon_{n-1},\dots,\varepsilon_1)$ is zero, or if $h^B \lambda_{n-1}^B(\varepsilon_{n-1},\dots,\varepsilon_1)$ is zero, there is nothing to compute. We may assume $\varepsilon_i=\alpha^{x_i}\beta^{y_i}$, with $x_i$ and $y_i$ as before. Then, the degree of $\lambda_{n-1}^B(\varepsilon_{n-1},\dots,\varepsilon_1)$ is 
\begin{align*}2-(n-1)+\sum\limits_{i=1}^{n-1} x_i+2y_i.\end{align*}
We apply $h^B$ and compose with $\varepsilon_n=\alpha^{x_n}\beta^{y_n}$ to get a chain map of degree $2-n+\sum\limits_{j=1}^n x_j+2y_j.$

We may recycle the arguments from the computation for $\operatorname{Ext}_B^\ast(\mathbb{L},\mathbb{L})$ to get that $m_n(g^\prime_n\varepsilon_n,\dots,\varepsilon_1)=0$ unless $n=\ell$. This works because, in the grading on $\operatorname{Ext}_\Lambda^\ast(\Delta,\Delta)$, elements of $\operatorname{Hom}_\Lambda(\Delta,\Delta)$ are homogeneous of degree 0, leaving the degree-based arguments unchanged. Again, there are three possible cases satisfying the requirement on the sum $\sum\limits_{j=1}^\ell x_j.$ If $x_j=0$ for all $ 1\leq j\leq \ell$, or $x_j=1$ for exactly one $1\leq j\leq \ell$, we may recycle the arguments from the case of $\operatorname{Ext}_B^\ast(\mathbb{L},\mathbb{L})$, arriving at a contradiction. Therefore, we consider the map
$\alpha \beta^{y_\ell} h^B \lambda_{\ell-1}^B(\alpha\beta^{y_{\ell-1}},\dots,\alpha\beta^{y_1}).$ We put $\gamma=\lambda_{\ell-1}^B(\alpha\beta^{y_{\ell-1}},\dots,\alpha\beta^{y_1})$ and $d_\gamma=\deg\gamma$. Then 
\begin{align*}\gamma=\vcenter{\xymatrix@=0.3cm{
	\dots\ar[r]&P^{d_\gamma+1}\ar[r]\ar[d]^-f& P^{d_\gamma} \ar[d]^-f\ar[r]& \dots \ar[r] &P^0	\\
	\dots\ar[r]&Q^1\ar[r]& Q^0
}},\quad \textrm{and}\quad h^B\gamma=\vcenter{\xymatrix@=0.3cm{
	\dots\ar[r]&P^{d_\gamma+3}\ar[d]^-0 \ar[r]& P^{d_\gamma+2}\ar[r]\ar[d]^-f&P^{d_\gamma+1}\ar[r]\ar[d]^-0&P^{d_\gamma}\ar[d]^-f \ar[r]& \dots \ar[r] & P^0	\\
	\dots\ar[r]&Q^4\ar[r]&Q^3\ar[r]& Q^2\ar[r]& Q^1
}},\end{align*}
which, composed with $\alpha\beta^{y_\ell}$, yields the map
$$\xymatrix@=0.3cm{
	P^{d_\gamma+2y_\ell+1}\ar[d]^-0\ar[r]&P^{d_\gamma+2y_\ell} \ar[d]^-f \ar[r]&\dots\ar[r]&P^{d_\gamma+1}\ar[r]\ar[d]^-0&P^{d_\gamma}\ar[d]^-f \ar[r]& \dots \ar[r] & P^0	\\
	Q^{2y_\ell+2}\ar[r] \ar[d]^-f&Q^{2y_\ell+1}\ar[r] \ar[d]^-1&\dots\ar[r]&Q^2\ar[r]& Q^1\\
	R^1\ar[r]&R^0
}$$
Put $\sigma_k=\sum\limits_{j=1}^k y_j$. Computing degrees of the involved maps and using our formula for the projective resolution of simple modules over $B$, we see that the above picture may be more precisely given as follows.
\begin{align*}\xymatrix@=0.3cm{
	\dots\ar[r]&	P_B(i+\ell(1+\sigma_\ell)+1)\ar[d]^-0\ar[r]&P_B(i+\ell(1+\sigma_\ell)) \ar[d]^-1 \ar[r]&\dots\ar[r]&P_B(i+\ell(1+\sigma_{\ell-1})+1)\ar[r]\ar[d]^-0&P_B(i+\ell(1+\sigma_{\ell-1}))\ar[d]^-1
	\\
	\dots\ar[r]&	P_B(i+\ell(2+\sigma_\ell)-1)\ar[r] \ar[d]^-f&P_B(i+\ell(1+\sigma_\ell))\ar[r] \ar[d]^-1&\dots\ar[r]&P_B(i+\ell(2+\sigma_{\ell-1})-1)\ar[r]& P_B(i+\ell(1+\sigma_{\ell-1}))\\
	\dots\ar[r]&	P_B(i+\ell(1+\sigma_\ell)+1)\ar[r]&P_B(i+\ell(1+\sigma_\ell))
}\end{align*}
We denote this composition by $\Gamma$ and put $d\coloneqq\deg \Gamma=2(1+\sigma_\ell)$. Consider the space of homogeneous maps of degree $\deg \Gamma$ from $P^\bullet$ to $R^\bullet$. Denote this space by $V^{d}$. It is clear from the above picture that $V^{d}$ has a basis given by maps $b_i$ of the following form.

\begin{align*}b_0=\xymatrix@=0.3cm{
	\dots\ar[r]&P_B(i+\ell(1+\sigma_\ell)+1)\ar[r] \ar[d]^- 0& P_B(i+\ell(1+\sigma_\ell)) \ar[d]^-1	\\
	\dots\ar[r]&P_B(i+\ell(1+\sigma_\ell)+1) \ar[r] & P_B(i+\ell(1+\sigma_\ell))
} \end{align*}
\begin{align*}b_1=\xymatrix@=0.3cm{
	\dots\ar[r]&P_B(i+\ell(1+\sigma_\ell)+1)\ar[r] \ar[d]^- 1& P_B(i+\ell(1+\sigma_\ell)) \ar[d]^-0	\\
	\dots\ar[r]&P_B(i+\ell(1+\sigma_\ell)+1) \ar[r] & P_B(i+\ell(1+\sigma_\ell))
}\end{align*}
Suppose this basis is $\{b_0,\dots, b_n\}$. We know that $\dim \operatorname{Ext}_B^d(L_B(i),L_B(i+\ell(1+\sigma_\ell)))=1.$ This space contains the extension $\varepsilon$ represented by the chain map
\begin{align*}\xymatrix@=0.3cm{
	\dots\ar[r]&	P_B(i+\ell(1+\sigma_\ell)+1)\ar[d]^-1\ar[r]&P_B(i+\ell(1+\sigma_\ell)) \ar[d]^-1	\\
	\dots\ar[r]&	P_B(i+\ell(1+\sigma_\ell)+1)\ar[r]&P_B(i+\ell(1+\sigma_\ell))
}\end{align*}
and, clearly, $\varepsilon=\sum\limits_{i=0}^n b_i$. Note that $\varepsilon$ is a basis for the homology of $V^d$. Since $\varepsilon$ is linearly independent from $b_1,\dots, b_n$, the set $\{\varepsilon,b_1,\dots, b_n\}$ is a basis of $V^d$. In this basis, we have $\Gamma=\varepsilon-b_1-b_3-\dots$ so that the projection of $\Gamma$ onto the homology is $\varepsilon$. This, in turn, implies that $p^B(\Gamma)$ is the extension $\beta^{\sigma_\ell+1}$. This yields the formula
\begin{align*}m_\ell(g^\prime \alpha\beta^{y_\ell},\dots, \alpha\beta^{y_1})=g^\prime \beta^{\sigma_\ell+1}=g^\prime m_\ell(\alpha\beta^{y_\ell},\dots,\alpha\beta^{y_1}).
\end{align*}
\subsection{Computing the box}
The results from \cite{kko} guarantee that there is an algebra $R$, Morita equivalent to $\Lambda$, which admits a \emph{regular exact Borel subalgebra} $\widehat{B}\subset R$. In our setup, $B\subset \Lambda$ is only an exact Borel subalgebra. We compute $\widehat{B}$ and $R$. We know that $\operatorname{Ext}_\Lambda^\ast(\Delta,\Delta)$ is given by
\begin{align*}\xymatrix{
		1\ar@<-0.5ex>[r]_{a_1} \ar@<.5ex>[r]^-{\alpha_1} \ar@/_2pc/[rrr]_-{\beta_1} & 2 \ar@/_2pc/[rrr]_-{\beta_2} \ar@<-0.5ex>[r]_{a_2} \ar@<.5ex>[r]^-{\alpha_2} & \dots \ar@<-0.5ex>[r] \ar@<.5ex>[r]& 1+\ell\ar@<-0.5ex>[r]_{a_{1+\ell}} \ar@<.5ex>[r]^-{\alpha_{1+\ell}} \ar@/_2pc/[rrrr]_-{\beta_{1+\ell}}& 2+\ell \ar@<-0.5ex>[r] \ar@<.5ex>[r]&\dots \ar@<-0.5ex>[r] \ar@<.5ex>[r]& 2\ell \ar@<-0.5ex>[r]_-{a_{2\ell}} \ar@<.5ex>[r]^-{\alpha_{2\ell}} & 1+2\ell \ar@<-0.5ex>[r] \ar@<.5ex>[r]&\dots \ar@<-0.5ex>[r] \ar@<.5ex>[r]& n
}\end{align*}
modulo the relations
\begin{align*}\alpha_{i+1}\alpha_i=0,\quad \alpha_{i+\ell}\beta_i=\beta_{i+1}\alpha_i\quad \alpha_{i+1}a_i=\beta_{i+1}a_i=0\quad  \textrm{and}\quad a_{i+\ell-1}\dots a_{i+1}a_i=0.\end{align*}
Next, we briefly describe the method we use to find $\widehat{B}$ and $R$, referring to \cite[Section~4.6]{bocsseat} and to \cite{kko}. We put $\mathbb{L}=L_\Lambda(1)\oplus\dots \oplus L_\Lambda(n)$. For any $n\geq 2$, we have a map
\begin{align*}
m_n: \left(\operatorname{Ext}_\Lambda^1(\Delta,\Delta)\right)^{\otimes_\mathbb{L} n}\to \operatorname{Ext}_\Lambda^2(\Delta,\Delta)
\end{align*}
because $m_n$ is of degree $2-n$ and the total degree of inputs is exactly $n$. Up to natural isomorphism, this gives a dual map
\begin{align*}
\mathbb{D}m_n: \mathbb{D}\operatorname{Ext}_\Lambda^2(\Delta,\Delta)\to \left(\mathbb{D}\operatorname{Ext}_\Lambda^1(\Delta,\Delta) \right)^{\otimes_\mathbb{L}n}.
\end{align*}
Summing these maps over all $n\geq 2$, we obtain the map
\begin{align*}
\sum \mathbb{D}m_n : \mathbb{D}\operatorname{Ext}_\Lambda^2(\Delta,\Delta)\to \medoplus_{n\geq 2} \left(\mathbb{D}\operatorname{Ext}_\Lambda^1(\Delta,\Delta)\right)^{\otimes_\mathbb{L} n}
\end{align*}
Then, we obtain $\widehat{B}$ as
\begin{align*}
\faktor{\medoplus_{n\geq 2} \left(\mathbb{D}\operatorname{Ext}_\Lambda^1(\Delta,\Delta)\right)^{\otimes_\mathbb{L} n}}{\left(\operatorname{im} \sum\mathbb{D}m_n\right)}=\widehat{B}.
\end{align*}
The extensions of degree 1 starting in $\Delta_\Lambda(i)$ are 
\begin{align*}
	\alpha_i &\in \operatorname{Ext}_\Lambda^1 (\Delta_\Lambda(i),\Delta_\Lambda(i+1)) \\
	a_{i+1}\alpha_i &\in \operatorname{Ext}_\Lambda^1(\Delta_\Lambda(i),\Delta_\Lambda(i+2)) \\
	&\vdots \\
	a_{i+\ell-1} \dots a_{i+1}\alpha_i &\in \operatorname{Ext}_\Lambda^1(\Delta_\Lambda(i),\Delta_\Lambda(i+\ell)).
\end{align*}
Then, $\widehat{B}$ has the following quiver.
\begin{enumerate}[label=$(\arabic*)$]
	\item Vertices are $1, \dots, n$.
	\item For each $i$, there are arrows
	\begin{align*}
	i\to i+1,\quad i\to i+2,\quad \dots\quad i\to i+\ell
	\end{align*}
	and no other arrows. 
\end{enumerate}
Since the only extensions of degree 2 are of the form $\varphi \beta$ where $\varphi\in \operatorname{Hom}_\Lambda(\Delta,\Delta)$, and these are uniquely obtained from higher multiplications by $m_\ell(\varphi \alpha,\dots,\alpha)=\varphi\beta$ we impose the relations $\varphi\alpha^\ell=0$. We draw the case $n=5,\ell=3$. The quiver, then, is
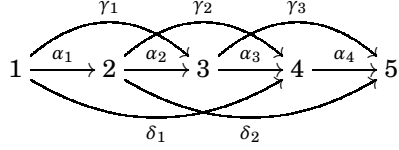
\vspace{.1cm}
\begin{align*}
\xymatrix{
1\ar[r]^-{\alpha_1} \ar@/^1.5pc/[rr]^-{\gamma_1} \ar@/_1.5pc/[rrr]_-{\delta_1}& 2\ar[r]^-{\alpha_2} \ar@/^1.5pc/[rr]^-{\gamma_2} \ar@/_1.5pc/[rrr]_-{\delta_2}& 3 \ar@/^1.5pc/[rr]^-{\gamma_3}\ar[r]^-{\alpha_3} & 4 \ar[r]^-{\alpha_4} & 5
}
\end{align*}
\\ \vspace{.1cm}
with the relations
\begin{align*}
\alpha_3\alpha_2\alpha_1=0,\quad \alpha_4\alpha_3\alpha_2=0,\quad \gamma_3\alpha_2\alpha_1=0.
\end{align*}
The Loewy diagrams of the indecomposable projective modules are then as follows.
	 \begin{align*}
	 P(5)\cong L(5),\quad P(4): \xymatrix@=0.3cm{4\ar[d]\\5},\quad P(3):\xymatrix@=0.3cm{ & & 3\ar[ld] \ar[d]\\
		& 4 \ar[ld] & 5 \\
		5},\quad P(2):\xymatrix@=0.3cm{
	4 &3\ar[d]\ar[l]& \ar[l]2 \ar[d]\ar[rd]\\
	& 5  & 4 \ar[d] & 5 \\
	 & & 5	
	},\quad P(1): \xymatrix@=0.3cm{3 & 2 \ar[ld] \ar[d]\ar[l]& 1 \ar[d] \ar[l] \ar[r] & 4 \ar[d]\ar[r] & 5 \\4 \ar[d]& 5 & 3\ar[d] \ar[ld]& 5\\
	5 & 4\ar[d]& 5\\
	  & 5
}\end{align*}
which implies that
\begin{align*}
R\otimes_{\widehat{B}} P_{\widehat{B}}(5)&=P_\Lambda(5)\\
R\otimes_{\widehat{B}} P_{\widehat{B}}(4)&=P_\Lambda(4)\\
R\otimes_{\widehat{B}} P_{\widehat{B}}(3)&=P_\Lambda(3)\oplus P_\Lambda(5)\\
R\otimes_{\widehat{B}} P_{\widehat{B}}(2)&=P_\Lambda(2)\oplus P_\Lambda(4) \oplus P_\Lambda(5)^{\oplus 2}\\
R\otimes_{\widehat{B}} P_{\widehat{B}}(1)&=P_\Lambda(1)\oplus P_\Lambda(3)\oplus P_\Lambda(4)^{\oplus 2}\oplus P_\Lambda(5)^{\oplus 3}
\end{align*}
so that 
\begin{align*}
R=\operatorname{End}_\Lambda\left( P_\Lambda(1) \oplus P_\Lambda(2) \oplus P_\Lambda(3)^{\oplus 2}\oplus P_\Lambda(4)^{\oplus 4} \oplus P_\Lambda(5)^{\oplus 7}\right)^{\operatorname{op}}.
\end{align*}
If an $A_\infty$-structure is not known, it is also possible to use the results by Conde in \cite{conde2020exact,conde2020quasihereditary} to find $R$.
\section*{Acknowledgement}
The author thanks the anonymous referee for the insightful and helpful comments which aided in improving the exposition of this work.
\newpage
\printbibliography
\end{document}